\documentclass[12pt,a4paper]{article}
\usepackage{eucal}
\usepackage{upgreek}
\usepackage{fullpage}
\usepackage{enumitem}

\usepackage{amsmath,amsfonts,amssymb,bef_alex,%
            datetime,todonotes,bm,relsize,tikz,mathrsfs}
\renewcommand{\ti}{{\times}}
\usepackage[colorlinks=true, allcolors=blue]{hyperref}

\numberwithin{equation}{section}
\numberwithin{figure}{section}

\DeclareMathOperator{\trace}{\mafo{tr}}

\newcommand{\HK}{\mathsf H\!\!\mathsf K}
\newcommand{\sfg}{\mathsf g}
\newcommand{\sfp}{\mathsf p}
\newcommand{\sfq}{\mathsf q}
\newcommand{\sfE}{\mathsf E}
\newcommand{\sfF}{\mathsf F}
\newcommand{\sfG}{\mathsf G}
\newcommand{\sfH}{\mathsf H}

\newcommand{\sfT}{\mathsf T}

\newcommand{\He}{\mathsf{He}}
\newcommand{\SHe}{{\mathsf S\!\mathsf H\mathsf{e}}}
\newcommand{\SHK}{{\mathsf S\!\mathsf H\!\!\mathsf K}}
\newcommand{\Ot}{\mathsf{Otto}}

\newcommand{\mfb}{\mathfrak b}
\newcommand{\mfc}{\mathfrak c}
\newcommand{\mfd}{\mathfrak d}
\newcommand{\mfG}{\mathfrak G}
\newcommand{\mfA}{\mathfrak A}
\newcommand{\mfP}{\mathfrak P}
\newcommand{\mfQ}{\mathfrak Q}
\newcommand{\LB}{\lambda_\rmB}
\newcommand{\uu}{u}
\newcommand{\HESS}{\mathop{\mafo{Hess}}\nolimits}
\newcommand{\spd}{\mathrm{spd}}
\newcommand{\target}{\uppi}

\newcommand{\tdfrac}{\tfrac} %

\newcommand{\EEE}{\color{black}}

\usepackage{algorithm}
\usepackage{algpseudocode}
\usepackage{algorithmicx}

\begin{document}

\title{Evolution of Gaussians in the 
\\Hellinger-Kantorovich-Boltzmann gradient flow%
}

\author{Matthias Liero\thanks{Weierstra\ss-Institut f\"ur Angewandte
   Analysis und Stochastik, Mohrenstr.\,39, 10117 Berlin, Germany}\,, Alexander
  Mielke$^*$, Oliver Tse\thanks{Department of Mathematics and Computer Science, Eindhoven University of Technology, 5600 MB Eindhoven, The Netherlands}, Jia-Jie Zhu$^*$}

\date{}
 
\maketitle

\begin{abstract}
This study leverages the basic insight that the gradient-flow equation associated with the relative Boltzmann entropy, in relation to a Gaussian reference measure within the Hellinger–Kantorovich (HK) geometry, preserves the class of Gaussian measures. This invariance serves as the foundation for constructing a reduced gradient structure on the parameter space characterizing Gaussian densities. We derive explicit ordinary differential equations that govern the evolution of mean, covariance, and mass under the HK–Boltzmann gradient flow. The reduced structure retains the additive form of the HK metric, facilitating a comprehensive analysis of the dynamics involved.

We explore the geodesic convexity of the reduced system, revealing that global convexity is confined to the pure transport scenario, while a variant of sublevel semi-convexity is observed in the general case. Furthermore, we demonstrate exponential convergence to equilibrium through Polyak–Łojasiewicz-type inequalities, applicable both globally and on sublevel sets. By monitoring the evolution of covariance eigenvalues, we refine the decay rates associated with convergence. Additionally, we extend our analysis to non-Gaussian targets exhibiting strong log-$\lambda$-concavity, corroborating our theoretical results with numerical experiments that encompass a Gaussian-target gradient flow and a Bayesian logistic regression application.

\medskip

  \smallskip

\noindent
\emph{MSC (2020) Subject Classifications:} Primary: 49Q22 (Optimal transport methods)
Secondary: 35Q49 (Transport equations)
\end{abstract}

\tableofcontents
\section{Introduction}
\label{se:intro}

Many computational problems in applied mathematics and statistics can be 
formulated as minimizing the relative Boltzmann entropy $\calH_\rmB$ over a 
subset $\mfA\subset \calM(\bbR^d)$ of finite measures on $\bbR^d$:
\begin{align}
    \min_{\mu \in \mfA} \calH_\rmB(\mu | \target)
      .
      \label{eq:forwardKL}
\end{align}
When $\mu$ and $\target$ are probability measures in $\calP(\bbR^d)$, then 
$\calH_\rmB(\mu|\target)$ is known as the Kullback-Leibler (KL) divergence 
between the measure $\mu$ and some \emph{target} measure $\target$.

For example, in the standard Bayesian inference framework
\cite{zellner1988optimal},
one considers a parametric model 
$p(\mathrm{Data}|\theta)$ with prior $p(\theta)$. The objective is then to infer 
the posterior distribution $\target(\theta) := p(\theta| \mathrm{Data})$ based 
on the data we have observed.
Except for the trivial cases, such as linear regression, the posterior $\target$
cannot be computed analytically and one must resort to computational approaches for
solving the variational inference problem as formulated in~\eqref{eq:forwardKL}.
The variational inference framework
\cite{jordanIntroductionVariationalMethods1999,
wainwrightGraphicalModelsExponential2008, blei_variational_2017} is based on 
the idea of restricting to measures $\mu$ to a specific class $\mfA$ of 
probability measures, such as Gaussian probability measures, allowing one 
to approximate the target measure $\target$ in a tractable manner.

Using the seminal work by \cite{Otto01GDEE}, one can view
solving the variational problem~\eqref{eq:forwardKL} over the space of 
probability measures $\calP(\bbR^d)$ as an Otto-Wasserstein gradient flow of 
the relative Boltzmann entropy \cite{wibisonoSamplingOptimizationSpace2018}. 
The associated gradient-flow equation is given by
\begin{align}
  \dot \rho &= 
  \Delta \rho 
  -
  \DIV\!\big( 
   {\rho}\nabla\log \target\big).
  \label{eq:WassersteinGradientFlow}
\end{align}
For the aforementioned Gaussian variational inference, one often restricts the
gradient flow equation \eqref{eq:WassersteinGradientFlow} to a submanifold of 
$\calP(\bbR^d)$, e.g., to the submanifold of Gaussian probability measures 
$\mfG_1\subset\calP(\R^d)$. Riemannian structures on $\mfG_1$ induced by the 
2-Wasserstein distance and the Fisher-Rao distance (also called Hellinger 
distance), respectively, have been
extensively studied in \cite{Modin2017}, where the author further provided
connections to entropy gradient flows and matrix decomposition problems. 
Other works of similar nature include
\cite{takatsu2011wasserstein, lambertVariationalInferenceWasserstein2022}, 
which investigated the restriction of \eqref{eq:WassersteinGradientFlow} to 
the family of Gaussian measures, and \cite{amariMethodsInformationGeometry2000,
amariInformationGeometry2021, chenGradientFlowsSampling2023} that focused on 
Fisher-Rao gradient flows.
In this work, we apply our findings from the Hellinger-Kantorovich ($\HK$) gradient flow of Gaussian measures, combining both the Bures-Wasserstein (Otto-Wasserstein but restricted to Gaussian measures) and Fisher-Rao (Hellinger restricted to Gaussian measures) structures, to solve the variational problem~\eqref{eq:forwardKL}. 
As such, this manuscript provides a survey and original research from a perspective different from that of other works currently available.
In particular, we derive reduced Onsager operators for the $\HK$-Gaussian to streamline the derivation of the reduced gradient flows over the Gaussian manifold.

\paragraph{$\HK_{\alpha,\beta}$-Boltzmann gradient system over Gaussian measures}
We consider the grad-\\
ient-flow equation for the gradient system $(\calM(\R^d),\calE, 
\HK_{\alpha,\beta})$, where $\calM(\R^d)$ is the set of all non-negative,
finite measures on $\R^d$ and the driving energy $\calE=\calH_\rmB 
(\cdot|\target)$ is the Boltzmann entropy
relative to an equilibrium measure $\target$, i.e.,
\[
\calH_\rmB(\rho|\target):= 
    \begin{cases}\displaystyle\int_{\R^d} 
    \lambda_\rmB\left(\frac{\dd \rho}{\dd \target}(x)\right) \target(\rmd x) 
    & \text{if $\rho\ll \target$,} \\
	+\infty & \text{otherwise,}
    \end{cases}
    \quad\text{with}\quad \lambda_\rmB(r) = r \log r - r +1.
\]
Following \cite{LiMiSa16OTCR,LiMiSa18OETP}, the 
$(\alpha,\beta)$-Hellinger-Kantorovich distance $\HK_{\alpha,\beta}$ 
is defined in terms of the (formal) Onsager operator given by
\[
\bbK^{\HK}_{\alpha,\beta}(\rho)\xi 
 = \alpha \bbK_\Ot(\rho)\xi + \beta \bbK_\mafo{He} (\rho)\xi
 := - \alpha \DIV\!\big(\rho\nabla \xi\big)+ \beta \rho \xi.
\]
Here, $\alpha$ and $\beta$ are two non-negative parameters measuring the
strength of the Wasserstein-Kantorovich part ($\alpha=1$, $\beta=0$) and 
the Hellinger part ($\alpha=0$, $\beta=4$).

The gradient-flow equation is given by $\dot\rho= - \bbK_{\alpha,\beta}(\rho)
\rmD \calE(\rho) $ and takes the explicit form
\begin{align}
\label{eq:I.HK.PDE}
\dot \rho = \alpha \DIV\!\big( \rho\nabla \log \bigl(\tdfrac{\rho}\target\bigr)\big) - \beta \rho \log\big(\tdfrac{\rho}\target\big).
\end{align}
We emphasize that $\rho\mapsto \calH_\rmB(\rho|\target)$ is not geodesically $\lambda$-convex in $(\calM(\R^d),\HK_{\alpha,\beta})$ whenever $\beta >0$, see
\cite{LiMiSa23FPGG}. Hence, the theory of evolutionary variational inequalities in the sense of \cite{AmGiSa05GFMS,LasMie22?EVIH} is not available. 

Instead, this paper focuses on a special subclass of solutions, namely that of scaled Gaussian distributions on $\R^d$. Namely, we consider the family of scaled Gaussian distributions on $\R^d$ given by 
\[
\Uppsi( \Sigma, m, \kappa)(x) := \kappa \,\sfG(\Sigma, m) (x) \quad \text{with}\quad 
\sfG(\Sigma,m;x) := \frac{\exp\!\big({-}\frac12 (x{-}m)\cdot \Sigma^{-1} (x{-}m)
  \big)}{\sqrt{\det(2\pi\Sigma)}}\, ,
\]
where $\kappa$ is the total mass, $m$ the mean value, and $\Sigma\in \R^{d\ti d}_\spd:= \{A\in \R^{d\ti d}\,|\, A=A^\top, A >0\}$ the covariance of the normalized Gaussian $\sfG(\Sigma,m)$, namely
\[
\kappa = \int_{\R^d} \Uppsi(\Sigma,m,\kappa) \dd x , \quad \kappa m= \int_{\R^d}
\! x \, \Uppsi(\Sigma,m,\kappa) \dd x, \quad 
\kappa \Sigma = \int_{\R^d}\! x{\otimes} x\, \Uppsi(\Sigma,m,\kappa) \dd x.
\]
We denote the manifold of all scaled and normalized Gaussians by 
\begin{align*}
\mfG &:= \bigset{ \kappa \,\sfG(\Sigma,m; \cdot) }{ \Sigma \in \R^{d\ti d}_\spd ,\
  m\in \R^d,\ \kappa>0 } \subset \calM(\R^d),\\
\mfG_1 &:= \bigset{ \sfG(\Sigma,m; \cdot) }{\Sigma \in \R^{d\ti d}_\spd,\ m\in \R^d}  
\subset \calP(\R^d).
\end{align*}
Setting $\mfP:=\R^{d\ti d}_\spd \ti \R^d \ti (0,\infty)$ and $\mfP_1:=\R^{d\ti d}_\spd \ti \R^d$, we see that the maps $\Uppsi\colon \mfP\to\mfG$, $\sfG\colon \mfP_1\to \mfG_1$ are smooth parametrizations of $\mfG$ and $\mfG_1$, respectively. Therefore, $\mfG$ is a smooth manifold of dimension $N_d:=1+d+d(d{+}1)/2=(d{+}2)(d{+}1)/2$, and $\mfG_1$ a smooth manifold of dimension $N_d-1$. 

The main observation is that if the equilibrium measure $\target$ is a scaled Gaussian, then the manifold of scaled Gaussians $\mfG$ is invariant under the evolution \eqref{eq:I.HK.PDE}. More precisely, suppose that $\target\in \mfG$ and $\rho(0)\in \mfG$. Then, the unique solution $t\mapsto \rho(t)$ to the gradient-flow equation \eqref{eq:I.HK.PDE} satisfies $\rho(t) \in \mfG$ for all $t>0$ (cf.\ Proposition~\ref{pr:Gauss.Sol}). This invariance allows us to derive a reduced gradient system $(\mfP,\sfE,\bbK_{\alpha,\beta}^\mathrm{red})$ on the parameter space $\mfP$ by restricting the gradient structure $(\calM(\R^d), \calE, \HK_{\alpha,\beta})$ to the Gaussian manifold $\mfG$ (cf.\ Theorems~\ref{th:ReductRiem} and \ref{th:ReductOnsag}).

\paragraph{Overview of the paper}
In Section \ref{se:DerivODE}, we first derive the ODE systems for the parameters
$\sfp=(\Sigma,m,\kappa)$ when the target $\target$ is a scaled Gaussian. This  derivation can be done by
simply matching coefficients when using the fact that all scaled Gaussians 
are given as the exponential of a quadratic function in $x\in \R^d$. 
A first observation is that the vector field  defining the ODE is still linear in the two parameters $\alpha$ and $\beta$ cf.\ \eqref{eq:ODE.mfP}. 

In Section \ref{se:DerivGradStructure}, we derive the gradient structure for the
ODE system. While the reduced energy functional is simply obtained by
evaluation of $\calH_\rmB$ on Gaussians, the reduction of the explicitly given
Onsager operator is nontrivial, and we follow the approach of
\cite[Sec.\,6.1]{MaaMie20MCRS}.
We obtain the useful formula for the reduced Onsager operator in the form of
\begin{align}
	\begin{aligned}
	\bbK_{\alpha,\beta}^\mathrm{red}(\sfp) &= \alpha \bbK_\Ot^\mathrm{red}(\sfp) + \beta \bbK_{\He}^\mathrm{red}(\sfp), \quad \text{where} \\
	\bbK_\Ot^\mathrm{red}(\Sigma,m,\kappa) &= \!\bma{ccc}\frac2{\kappa}\big(
\Box \Sigma + \Sigma \Box\big) &0&0 \\ 0 & \frac1\kappa\,\Box & 0\\ 0&0& 0 \ema, \\
\bbK_{\He}^\mathrm{red}(\Sigma,m,\kappa) &= \!\bma{ccc}\frac2\kappa\, \Sigma \Box
\Sigma &0&0 \\ 0 & \frac1\kappa\,\Sigma \Box& 0\\ 0&0& \kappa \Box\ema.
	\end{aligned}
\end{align}
This crucial formula allows us to easily make various calculations for the rest of the paper.

In Section \ref{se:GaussProba}, we consider the gradient systems 
\[ 
 (\calP(\R^d),\calH_\rmB(\cdot|\target) , \SHK_{\alpha,\beta}) ,
\]
where $\calP(\R^d)$ denotes the set of all probability measures on $\R^d$. The
spherical Hellinger-Kantorovich distance $\SHK_{\alpha,\beta}$ is  given in
terms of the Onsager operator 
\[
\bbK^{\SHK}_{\alpha,\beta}(\rho)\eta= - \alpha \DIV\!\big(\rho\nabla 
\eta\big) + \beta \rho\big( \eta -{\ts \int_{\R^d} \rho \eta\dd x}\big),
\]
which leads to the associated gradient-flow equation
\begin{align}
 \label{eq:I.SHK.PDE}
\dot \rho &= \alpha \DIV\!\big( \nabla \rho + \frac{\rho}{\target}\nabla\target\big) 
 - \beta \rho\,\Big( \!\log\big(\frac{\rho}\target\big)
 -\int_{\R^d}\rho\log\big(\frac{\rho}\target\big) \dd x  \Big).
\end{align}
The solutions of this equation are closely related to the gradient-flow
equation \eqref{eq:I.HK.PDE} for the gradient system $
(\calM(\R^d),\calH_\rmB(\cdot|\target) , \HK_{\alpha,\beta})$, see
\cite{mielke2025hellinger}: 
 If $t \mapsto \uu(t)$ solves \eqref{eq:I.HK.PDE}, then $t\mapsto \rho(t)=
  \tdfrac1{z(t)} \uu(t)$ with $z(t)=\int_{\R^d} \uu(t,x)\dd x$ solves
  \eqref{eq:I.SHK.PDE}. Moreover, if $t\mapsto \rho(t)$ solves
  \eqref{eq:I.SHK.PDE}, then $t \mapsto \uu(t)=\kappa(t) \rho(t)$ solves
  \eqref{eq:I.HK.PDE} for suitable functions $t\mapsto \kappa(t)$.

This relation shows that \eqref{eq:I.SHK.PDE} also preserves the Gaussian form. The reduced
evolution for $(\Sigma,m)$
directly corresponds to the evolution of scaled Gaussians. 

In Section~\ref{se:GeodesicConvexity},
we investigate geodesic $\uplambda$-convexity for the finite-dimensional 
gradient system $(\mfP_1, \sfE_1, \bbK^\mathrm{red}_{\alpha,\beta})$ associated 
with normalized Gaussian measures and relative Boltzmann entropy
$\sfE_1(\sfp)= \sfH(\sfp| \Gamma,n)$, where $\sfp=(\Sigma,m)$. Geodesic convexity 
is characterized via the metric Hessian $\HESS_{\bbK}\sfF$, with the key condition
\[
\langle \xi, \HESS_{\bbK} \sfF(\sfq)\, \xi \rangle \geq \uplambda \langle \xi, \bbK(\sfq)\, \xi \rangle,
\]
expressing $\uplambda$-convexity along geodesics. We show that global geodesic 
convexity holds \emph{if and only if} $\beta = 0$ and $\uplambda \leq \alpha\,
\nu_{\min}(\Gamma^{-1})$, i.e., the case of pure Otto-Wasserstein.
For $\beta > 0$, we show that geodesic convexity fails for the Gaussian 
probability measures. However, 
\emph{sublevel geodesic semi-convexity}~\eqref{eq:SublGLCvx} still holds 
for the relative Boltzmann entropy $\sfE_1(\sfp)$.

In Section~\ref{se:LongTime}, we analyze the \emph{long-time decay behavior} 
of solutions to the $\HK_{\alpha,\beta}$-Boltzmann gradient flow over 
Gaussian measures. Our first key tool is the \emph{Polyak--\L ojasiewicz (P\L)}
functional inequality, which enables exponential convergence rates even in 
the absence of geodesic sublevel semi-convexity.

In Section~\ref{su:PolLojEstim}, we establish both \emph{global} and \emph{sublevel P\L} inequalities for the gradient flow energy functional $\sfE_1$, using a decomposition into \emph{covariance} and \emph{mean} parts. This yields exponential decay estimates of the form:
\[
\sfE_1(\sfp(t)) \leq \ee^{-\mfc_\mafo{cov}t} \sfH_\mafo{cov}(\Sigma(0)) + \ee^{-\mfc_\rmm t} \sfH_\rmm(m(0)),
\]
with rates $\mfc_\mafo{cov}, \mfc_\rmm$ that depend explicitly on $\alpha, \beta$, and the energy level $\sfE_1(\sfp(0))\leq E$.

A \emph{refined analysis} follows, where the decay is sharpened by tracking the evolution of eigenvalues of the normalized covariance matrix $B = \Gamma^{-1/2}\Sigma\Gamma^{-1/2}$. This leads to improved decay rates:
\[
\sfH_\mafo{cov}(\Sigma(t)) \le C \, \ee^{-\upnu_\mafo{cov} t}, \quad
\sfH_\rmm(m(t)) \le C \, \ee^{-\upnu_\rmm t},
\]
with rates $\upnu_\mafo{cov} = 2\alpha \nu_\mafo{min}(\Gamma^{-1}) + \beta$ and $\upnu_\rmm = 2\alpha \nu_\mafo{min}(\Gamma^{-1}) + 2\beta$,
which are independent of the energy level (initial condition) $E$.
Here, $\alpha$ is again the scaling parameter of the Otto-Wasserstein geometry and $\beta$ that of the Hellinger (Fisher-Rao) geometry. We see that while the Otto-Wasserstein part of the decay rate is affected by $\nu_\mafo{min}(\Gamma^{-1})$, the Hellinger part is not.

Finally, the results are generalized to \emph{log-$\uplambda$-concave target 
measures} $\target$, which are not necessarily scaled Gaussians. In this case,
exponential decay is still achieved via energy-dissipation inequalities 
involving the metric slope and Gr\"onwall's lemma, yielding a exponential 
decay result, which matches standard results for the pure Wasserstein and 
Hellinger settings in the 
literature~\cite{chenGradientFlowsSampling2023, 
lambertVariationalInferenceWasserstein2022}.

In Section~\ref{sec:NumericalAlgorithms}, we introduce a discrete-time gradient 
descent algorithm tailored to the $\HK_{\alpha,\beta}$-Boltzmann gradient flow 
on Gaussian measures. This method approximates the continuous-time evolution 
given by equation~\eqref{eq:HK.Gauss.mass-shape-ode-all} by interleaving three 
key updates: (i) a Fisher-Rao step (spherical Hellinger over Gaussian probability measures), which updates the covariance and mean using Euler discretization of system~\eqref{eq:Fisher-Rao-Gaussian}; (ii) a Bures-Wasserstein step (Otto-Wasserstein over Gaussian probability measures), which discretizes the Bures-Wasserstein ODE~\eqref{eq:Bures-Wasserstein-ODE}; and (iii) a mass update step induced by the Hellinger dissipation. The resulting Algorithm~\ref{alg:JKO} incorporates Monte Carlo estimation to handle non-Gaussian target measures and is applied to both Gaussian target flow and Bayesian logistic regression. Numerical results show interesting convergence behavior when compared to pure Wasserstein or Fisher–Rao flows.
For example, when the initial measure is far away from the target, the Bures-Wasserstein iterations make quick progress, while the Fisher-Rao iterations stagnate due to inaccurate covariance estimation. However, asymptotically, Fisher-Rao iterations enjoy faster convergence near equilibrium.

\section{ODE system for \texorpdfstring{$(\Sigma,m,\kappa)$}{(Si,m,ka)}}
\label{se:DerivODE}

\subsection{Working with a simpler parametrization} 
\label{su:SimplerCoord}

We begin our discussion with a simpler parametrization for $\mfG$ that avoids the inversion of the covariance matrix. Observe that scaled Gaussians can be alternatively expressed as
\[
\rho(x)= \Upphi(A,b,c)(x) :=\exp(B_\sfq(x)), \quad \text{where}\quad B_\sfq(x)=
c+ b\cdot x-\frac12 x\cdot Ax .
\]
Note that by using the simple coordinates $\sfq=(A,b,c)\in \mfQ:=\R^{d\ti d}_\spd\ti \R^d \ti \R$ instead of $\mfP$, we are no longer using the \emph{usual} variables $(\Sigma,m,\kappa)\in\mfP$ for parametrizing Gaussians. Nevertheless, the parametrization $\Upphi: \mfQ \to \mfG$ is still a diffeomorphism.

To return to $(\Sigma, m,\kappa)\in \mfP \EEE$, 
we have the transformation $\sfT:\mfQ\to \mfP$ with
\begin{equation}
  \label{eq:trafoQP}
  \Sigma = A^{-1}, \quad m= A^{-1} b, \quad \kappa = \frac{(2\pi)^{d/2}}{ (\det
    A)^{1/2}} \, \exp\big( c+\frac12 b\cdot A^{-1} b\big). 
\end{equation}
The parametrization
$(-\frac12 A, b)=(-\frac12 \Sigma^{-1}, \Sigma^{-1} m)$
is also referred to as the \emph{natural parametrization} of the (normalized) Gaussian measures \cite{amariMethodsInformationGeometry2000}.
Applying the transformation $\rho(t,x)=\ee^{B(t,x)}$ to
\eqref{eq:I.HK.PDE} with $\target(x)= \ee^{\ol B(x)}$, we obtain
\begin{equation}
  \label{eq:Hopf.GFE}
  \dot B = \alpha\big( \Delta(B{-}\ol B) + \nabla B\cdot \nabla(B{-} 
   \ol B)\big) -   \beta\,(B- \ol B). 
\end{equation}
We can easily see that for bilinear $\ol B$ and bilinear initial condition $B(0,\cdot)$, there is a bilinear solution given explicitly by
\[
B(t,x)=B_{\sfq(t)}(x) = c(t) + b(t)\cdot x -\frac12 x \cdot A(t) x ,
\]
where $t\mapsto \sfq(t) =(A(t),b(t),c(t)) \in \mfQ$ satisfies the following system of ODEs:
\begin{subequations}
\label{eq:ODEs.mfQ}
\begin{align}
\label{eq:ODEs.mfQ.A}
\dot A&= \alpha \big(\, \ol A A + A\ol A - 2 A^2 \big) + \beta( \ol A{-} A),
\\
\dot b&=  \alpha \big( A \ol b + \ol A b - 2 A b \big) + \beta \,( \ol b - b),
\\
\dot c&= \alpha\big( \trace(\, \ol A{-} A) + b\cdot( b{-} \ol b)\big) 
 + \beta \,( \ol c - c)
\end{align}
\end{subequations}
subject to the initial condition $\sfq(0)=\sfq_0=(A_0,b_0,c_0)\in \mfQ$.

The global existence of the system \eqref{eq:ODEs.mfQ} follows from the standard theory of nonlinear differential equations. For completeness, a discussion on the well-posedness of \eqref{eq:ODEs.mfQ} is performed in Appendix~\ref{se:ODEs.mfQ}. We summarize the current findings in the following. 

\begin{proposition}[Gaussian solutions for \eqref{eq:I.HK.PDE}]
\label{pr:Gauss.Sol} The unique solution $t\to \rho(t)$ of equation \eqref{eq:I.HK.PDE}
with initial condition $\rho(0,\cdot)= \Upphi(\sfq_0) \in \mfG\subset
\calM(\R^d)$  is given by $\rho(t,\cdot)= \Upphi(\sfq(t))$, where $t\mapsto \sfq(t) \in
\mfQ$ is the unique solution of \eqref{eq:ODEs.mfQ} with $\sfq(0)=\sfq_0\in\mfQ$. 
\end{proposition}

\subsection{Returning to standard coordinates} 
\label{su:ReturnCoord}

We now use the transformation \eqref{eq:trafoQP} to generate the ODEs in the parameter space $\mfP$. 

\begin{proposition}[Induced GFE on $\mfP$]
The Gaussian solutions 
for \eqref{eq:I.HK.PDE} are uniquely determined in the form $\rho(t) =
\kappa (t)\sfG(\Sigma(t),m(t))$, where $t\mapsto \sfp(t)=(\Sigma(t), m(t), \kappa(t))\in\mfP$ solves
the following system of ODEs:
\begin{subequations}
\label{eq:ODE.mfP}
\begin{align}
\dot \Sigma&= \alpha \big(2 I - \Gamma^{-1}\Sigma - \Sigma \Gamma^{-1} \big) 
 + \beta\big( \Sigma {-}\Sigma \Gamma^{-1} \Sigma \big),\label{eq:ODE.mfP.sigma}
\\[0.3em]
\dot m&=  \alpha \, \Gamma^{-1} (n - m ) +
\beta \,\Sigma \Gamma^{-1}(n- m),
\\
\dot \kappa &=  \beta \,\frac\kappa2 \left( \trace( I{-}
  \Gamma^{-1}\Sigma) -(n{-}m)\cdot \Gamma^{-1}(n{-}m)
  -2\log\Big(\frac\kappa\varkappa\Big) +
  \log\Bigl(\frac{\det\Sigma}{\det\Gamma}\Bigr) \right),
\end{align}
where $\target=\varkappa \sfG(\Gamma,n)$ is the target measure. 
\end{subequations}
\end{proposition}
Note that in the equation for $\dot\kappa$, the contribution proportional to 
$\alpha $ vanishes as the transport term is mass preserving. \smallskip

\noindent\begin{proof}
The equations for $\dot\Sigma$ and $\dot m$ follow by straightforward
substitution. The equation for $\dot\kappa$ is obtained by observing the relation $\kappa = \sqrt{\det (2\pi\Sigma)} \exp( c+\frac12 b\cdot
m) $ and $D (\det)(\Sigma)[A] = \trace\big(\mafo{cof}(\Sigma) A\big) = (\det
\Sigma) \,\trace\big(\Sigma^{-1}A\big)$, giving 
\[
\dot \kappa = \frac{\kappa}{2} \,\Big(\trace\big(\Sigma^{-1} \dot \Sigma\big) +
\dot c + \big(\dot b\cdot m + b\cdot \dot m \big)  \Big). 
\] 
Inserting the established relations gives the desired result. 
\end{proof}
Several remarks are in order.
\begin{remark}[Explicit solutions]
\label{re:Explicit.A(t)}
It turns out that it is possible to solve equation~\eqref{eq:ODEs.mfQ.A} and, consequently, \eqref{eq:ODE.mfP.sigma} explicitly. 

Indeed, consider the transformation $A(t)=\ol A + \ee^{-Ct}B(t)\ee^{-C
  t}$ with $C= \alpha \ol A+\frac\beta2 I$, which yields the equation $\dot B = -2\alpha B \ee^{-2Ct} B$ for $B$. Setting $D=B^{-1} $ leads to $\dot D= 2\alpha \ee^{-2Ct}$. Solving for $D$ then gives
\[
	D(t)= D(0)+\alpha C^ {-1}(I{-}\ee^{-2Ct})=  D(0)+\alpha C^{-1} - \alpha \ee^{-Ct}C^{-1} \ee^{-Ct}.
\]
From this, we obtain the explicit solution $A$ to \eqref{eq:ODEs.mfQ.A} 
in the form
\[
	A(t) = \ol A + \ee^{-C t}\big( (A_0{-}\ol A)^{-1} {+} \alpha C^{-1} {-} \alpha \ee^{-Ct}C^{-1} \ee^{-Ct}\big)^{-1} \ee^{-Ct} \quad \text{with}\quad C= \alpha \ol A+\tfrac\beta2 I,
\]
and therefore also an explicit form for $\Sigma=A^{-1}$.

When $\alpha\beta=0$, the expressions simplify: In the pure transport case ($\beta=0$), we have
\[
\Sigma(t)= \Gamma + \ee^{-\alpha t \Gamma^{-1}}\bigl(\Sigma_0- \Gamma\bigr)
\ee^{- \alpha t \Gamma^{-1}},\quad A(t) = \ol A \Big( \ol A +
\ee^{-\alpha t \ol A} \big( \ol A \,A_0^{-1}\ol A - \ol A\big) \ee^{-\alpha t
  \ol A} \Big)^{-1} \ol A.
\]
While in the pure growth case ($\alpha = 0$), we have 
\[
A(t) = \ol A + \ee^{-\beta t}( A_0- \ol A) \quad \text{or} \quad 
\Sigma(t) = \Gamma\Big( \Gamma + \ee^{-\beta t} \big( \Gamma \Sigma_0^{-1}
\Gamma - \Gamma\big)\Big)^{-1} \Gamma. 
\]
\end{remark}

\begin{remark}[Long-time asymptotics]\label{re:Explicit.A(t).decay}
Observe that $A(t)$ converges to $\ol A$ with
\[
	A(t)= \ol A + \ee^{-C t} B_* \ee^{-Ct} + O(|\ee^{-Ct}|^4)\quad 
    \text{for $t\to \infty$.}
\]
In the later sections, we are interested in the long-time asymptotics for 
$\Sigma= A^{-1}$ with $\ol A=\Gamma^{-1}$. Hence, we find the expansion 
\begin{equation}
    \label{eq:SigmaConverges}
    \Sigma(t) = \Gamma - \Gamma \ee^{-C t} B_* \ee^{-C t} \Gamma + 
    O(|\ee^{-Ct}|^4) \quad \text{for $t\to \infty$}, \quad \text{with}\quad C=\alpha\Gamma^{-1}+\tfrac\beta2I.
\end{equation}
In particular, we have the decay estimate $|\Sigma(t){-} \Gamma|\leq C 
\ee^{-\nu_\mafo{opt} t} $ with the optimal decay rate 
$\nu_\mafo{opt} = \beta+2\alpha \nu_{\min}(\Gamma^{-1}) >0$ given as smallest 
eigenvalue of $2\alpha \Gamma^{-1} {+} \beta I$. 
\end{remark}

\section{Reduced gradient system}
\label{se:DerivGradStructure}

\subsection{General reduction of gradient structures}

Reducing a gradient system to an invariant submanifold is trivial when
the Riemannian tensor is given. We summarize the result in the following
statement. 

\begin{theorem}[Reduction of Riemannian GS]\label{th:ReductRiem}
Let the gradient system $(\calM,\calE,\bbG)$ be given in terms of a family of symmetric and positive definite Riemannian operators $\bbG(u): \rmT_u\calM\to \rmT^*_u \calM$. 
Assume that the smooth submanifold $\calN\subset \calM$ is invariant under the 
gradient flow $\bbG(u)\dot u =-\rmD\calE(u) \in \rmT^*_u \calM$ (where 
$\rmD \calE$ denotes the classical Fr\'echet derivative). 

Then, the restricted flow on $\calN$ is again a gradient system and it is given 
by the reduced gradient system $(\calN,\calE|_\calN, \bbG^\mathrm{red})$, where 
$\bbG^\text{red}$ is simply the restriction of $\bbG$ to $\calN$, namely
\[
\forall\,w\in \calN,\ \forall\, v \in \rmT_w\calN: \quad 
\langle \bbG^\mathrm{red}(w)v,v \rangle_{\rmT_w\calN} = \langle \bbG(w)v,v
\rangle_{\rmT_w\calM}. 
\]
\end{theorem}

To clarify the situation with the next result, it is helpful to introduce a 
notation for the embedding of $\calN$ into $\calM$, namely $j:\calN\to \calM$. 
Now, $\rmD j(w): \rmT_w \calN \to \rmT_{j(w)}\calM$ and $\rmD j(w)^* : 
\rmT^*_{j(w)} \calM \to \rmT^*_w\calN$ (again $\rmD$ denotes Fr\'echet 
derivatives). With this, $\bbG^\text{red}$ can be expressed more explicitly 
in the form  
\[
\bbG^\text{red}(w) = \rmD j(w)^* \bbG(j(w)) \rmD j(w): \rmT_w \calN \to
\rmT^*_w \calN. 
\]
The situation is slightly more involved if the Onsager operators 
$\bbK(u)=\bbG(u)^{-1}$ are given instead, and the inversion is not desirable. 
On the formal level, the reduction means taking a Schur complement, 
see Remark \ref{re:Schur}. 

\begin{theorem}[Reduction of Onsager GS]\label{th:ReductOnsag} 
Let the gradient system $(\calM,\calE,\bbK)$ be given in terms of a family 
of symmetric and positive definite Onsager operators $\bbK(u): \rmT^*_u 
\calM\to \rmT_u \calM$. Assume that the smooth submanifold $\calN\subset 
\calM$ is invariant under the gradient flow $\dot u =-\bbK(u)\rmD\calE(u)$. 

Then, the restricted flow on $\calN$ is again a gradient system, and it is 
given by the reduced gradient system $(\calN,\calE|_\calN, \bbK^\mathrm{red})$, 
where the reduced Onsager operator is given by 
\[
\forall\,w\in \calN,\ \forall\, \eta \in \rmT^*_w\calN: \quad 
\langle \eta, \bbK^\mathrm{red}(w)\eta\rangle_{\rmT_w\calN} = \inf\bigset{\langle
\xi,\bbK(j(w))\xi\rangle_{\rmT_{j(w)}\calM}} { \rmD j(w)^*\xi = \eta},
\]
where $j:\calN\to \calM$ is the embedding mapping. 
\end{theorem}
\begin{proof}
The result is the content of the theory developed in \cite[Sec.\,6.1]{MaaMie20MCRS}. 
\end{proof}

In both results above, we can use a parameter manifold $\mfP$ instead of
the submanifold $\calN$, i.e., we have an invertible map $\Uppsi: \mfP \to \calN
\subset \calM$. We can then define the reduced gradient system with the parameter
manifolds $\mfP$ as the state space. The reduced energy functional is given by 
$\sfE=\calE\circ \Uppsi; \ \sfp\mapsto \calE(\Uppsi(\sfp))$. 

In the Riemannian case, we obtain 
the reduced gradient system
\[
(\mfP,\sfE, \bbG^\text{red}) \quad\text{with}\quad \bbG^\text{red}(\sfp)=\rmD \Uppsi(\sfp)^*
\bbG(\Uppsi(\sfp)) \rmD \Uppsi(\sfp): \rmT_\sfp\mfP\to \rmT^*_\sfp\mfP.
\]
For the Onsager case, we obtain the corresponding reduced gradient system in the form 
\begin{align*}
(\mfP,\sfE, \bbK^\mathrm{red}) \quad \text{with } 
\langle \eta,\bbK^\mathrm{red}(\sfp)\eta\rangle_{\rmT_\sfp\mfP} = \inf\bigset{\langle
  \xi,\bbK(\Uppsi(\sfp))\xi\rangle_{\rmT_{\Uppsi(\sfp)}\calM} }{ \rmD\Uppsi(\sfp)^*\xi=\eta} . 
\end{align*}

The calculation of $\bbK^\mathrm{red}$ can be performed using the standard Lagrange theory for constraint minimization: For fixed $\sfp \in \mfP$ and $\eta \in
\rmT^*_\sfp\mfP$, one defines the  Lagrange function
\[
\mathfrak{L}(\xi,\lambda) := \frac{1}{2}\big\langle \xi,
  \bbK(\Uppsi(\sfp)) \xi \big\rangle_{\rmT_{\Uppsi(\sfp)}\calM} + \langle \rmD
  \Uppsi(\sfp)^* \xi {-}\eta,\lambda\rangle_{\rmT_\sfp\mfP}.
\]
The optimal $\xi$ can be found as the critical point of $\mathfrak L$, i.e., by solving the reduction equation   
\begin{equation}
  \label{eq:AbstrRedEqn}
\bbK(\Uppsi(\sfp)) \xi = \rmD\Uppsi(\sfp) \lambda \in \rmT_{\Uppsi(\sfp)}\calM \quad \text{and} \quad 
\rmD\Uppsi(\sfp)^* \xi=\eta \in \rmT^*_\sfp\mfP  
\end{equation}
for $(\xi,\lambda) \in \rmT^*_{\Uppsi(\sfp)}\calM\ti \rmT_\sfp\mfP$, 
where $\lambda \in \rmT_\sfp\mfP$ is the Lagrange multiplier. 

\medskip
In summary, we obtain the following result, which will be used to derive the reduced gradient structure on the parameter space $\mfP$.

\begin{proposition}[Formula for reduced Onsager operator]
\label{pr:ReducedOnsager}
Under the above notations and assumptions, the reduced Onsager operator $ \bbK^\mathrm{red}$ is given by  
\begin{subequations}
 \label{eq:ReducOnsager}
\begin{equation}
    \label{eq:ReducOnsager.a}
    \langle \eta , \bbK^\mathrm{red}(\sfp) \eta\rangle 
      = \langle \eta, \ol\lambda \rangle 
      = \langle \ol\xi, \bbK(\Uppsi(\sfp)) \ol\xi \rangle,
\end{equation}
where $(\ol\xi,\ol\lambda)\in \rmT^*_{\Uppsi(\sfp)}\calM\ti \rmT_\sfp\mfP$ is a solution of 
\begin{equation}
    \label{eq:ReducOnsager.b}
\bma{cc} \bbK(\Uppsi(\sfp)) & -\rmD\Uppsi(\sfp)\\ -\rmD\Uppsi(\sfp)^* &0\ema \binom{\ol\xi}{\ol\lambda} = \binom{0}{-\eta}. 
\end{equation}
\end{subequations}
\end{proposition}

\begin{remark}[Reduction in terms of the Schur complement] 
\label{re:Schur}
As a special case of formula \eqref{eq:ReducOnsager}, we consider the simplified case of $X=X_1\ti X_2$ with $X_1\ti\{0\}$ being an invariant subspace for the gradient system $(X,\calE,\bbK)$, where 
\[
	X=X_1\ti X_2 \quad \text{and} \quad \bbK(u) = \bma{@{}cc@{}} K_{11}(u)& K_{12}(u) \\ 
	K_{21}(u)& K_{22}(u) \ema\!{:}\,X_1^*\ti X_2^* \to X_1\ti X_2.
\]
Writing $u=(u_1,u_2)$ and $\rmD \calE(u)=(\rmD_1 \calE(u),\rmD_2 \calE(u))$, the gradient-flow equation reads
\begin{align*}
\dot u_1&= -\big(K_{11}(u)\rmD_1 \calE(u){+} K_{12}(u)\rmD_2 \calE(u)\big), \quad 
\dot u_2= -\big(K_{21}(u)\rmD_1 \calE(u){+} K_{22}(u)\rmD_2 \calE(u)\big).
\end{align*}
The invariance condition of $X_1\ti\{0\}$ means
\[
\dot u_2=0 = K_{21}(u_1,0) \rmD_1 \calE(u_1,0) + K_{22}(u_1,0) \rmD_2 \calE(u_1,0),
\]
which gives $\rmD_2 \calE(u_1,0)= - K_{22}(u_1,0)^{-1}  K_{21}(u_1,0) \rmD_1
\calE(u_1,0)$.  Inserting this relation
into the equation for $\dot u_1$ and writing $\sfE(u_1)=\calE(u_1,0)$, we obtain $\dot u_1 = - \bbK^\mathrm{red}(u_1) \rmD\sfE(u_1)$ with
\[
	\bbK^\mathrm{red}(u_1) = K_{11}(u_1,0) - K_{12}(u_1,0) \big(K_{22}(u_1,0)\big)^{-1} K_{21}(u_1,0),
\]
i.e., $\bbK^\mathrm{red}(u_1)$ is the Schur complement of $\bbK(u_1,0)$
when eliminating the $X_2$ component. 
\end{remark}

\subsection{Reduced Onsager operators}
\label{su:Reduct.HK}

If we apply the abstract theory from above to the Gaussian family 
$\Uppsi(\Sigma,m,\kappa)=\kappa \sfG(\Sigma,m)$, we can again take advantage 
of the property that the Onsager operator $\bbK_{\alpha,\beta}(\rho)$ is the 
linear combination $\alpha \bbK_\Ot(\rho) + \beta \bbK_\mafo{He}(\rho)$, where 
both operators $ \bbK_\Ot(\rho)$ and $\bbK_\mafo{He}(\rho)$ are linear in 
$\rho$. Moreover, the Gaussian family has the property that 
$\rmD \Uppsi(\sfp)[\wh \sfp]$ is always the product of a Gaussian and a 
quadratic polynomial in $x \in \R^d$. This characteristic is exactly mirrored 
in the Onsager operator $\bbK_{\alpha,\beta}$ in the sense that 
$\bbK_{\alpha,\beta}(\Uppsi(\sfp)) \xi$ is again a product of a Gaussian and 
a quadratic polynomial if $\xi$ is a quadratic polynomial. These two facts 
will lead to the surprising result that the reduced Onsager operator is again 
a linear combination, namely
\[
	\bbK^\mathrm{red}(\sfp)=\alpha \,\bbK^\mathrm{red}_\mafo{tra}(\sfp) 
                      + \beta \,\bbK^\mathrm{red}_\mafo{He}(\sfp), 
\]
see Theorem~\ref{th:ReducedGradStr} below. 

Before doing the calculation, we provide some details on the 
derivative $\rmD\Uppsi$ in the following remark,
which will also be useful in Section~\ref{se:ExtenNonGauss}, where we consider non-Gaussian target measures $\target\in\calM(\bbR^d)$.

\begin{remark}\label{rem:psi-derivative}
	Recall the map $\Uppsi(\Sigma,m,\kappa) = \kappa \sfG(\Sigma,m)$. Simple computations yield
	\begin{align*}
		\rmD_\Sigma \sfG(\Sigma,m)[\wh \Sigma] &= \frac{1}{2}\left[ \trace\bigl[\wh\Sigma \bigl(\Sigma^{-1}(x-m)\otimes\Sigma^{-1} (x-m)\bigr)\bigr] - \trace\bigl[\wh\Sigma\Sigma^{-1}\bigr] \right]\sfG(\Sigma,m), \\
		\rmD_m \sfG(\Sigma,m)[\wh m] &= \langle \wh m,\Sigma^{-1}(x-m)\rangle\, \sfG(\Sigma,m),
	\end{align*}
	which, due to the identities
	\begin{align*}
		\nabla_x \sfG(\Sigma,m;x) &= -\Sigma^{-1}(x-m)\,\sfG(\Sigma,m),  \\
		\nabla_x^2 \sfG(\Sigma,m;x) &= \Bigl[\Sigma^{-1}(x-m)\otimes \Sigma^{-1}(x-m) - \Sigma^{-1}\Bigr]\sfG(\Sigma,m),
	\end{align*}
	may be equivalently expressed as
	\begin{align*}
		\rmD_\Sigma \sfG(\Sigma,m)[\wh \Sigma] &= \frac{1}{2}\trace\bigl[\wh\Sigma \nabla_x^2 \sfG(\Sigma,m)\bigr] \\
		\rmD_m \sfG(\Sigma,m)[\wh m] &= -\langle \wh m,\nabla_x \sfG(\Sigma,m)\rangle .
	\end{align*}
	Consequently, we obtain
	\[
		\rmD\Uppsi(\Sigma,m,\kappa)[\wh \Sigma,\wh m,\wh\kappa] =
			\frac{1}{2}\trace\bigl[\wh\Sigma \nabla_x^2 \Uppsi(\Sigma,m)\bigr] 
			-\langle \wh m,\nabla_x \Uppsi(\Sigma,m)\rangle +
			\wh\kappa \sfG(\Sigma,m).
	\]
\end{remark}

We now return to the derivation of the reduced Onsager operator $\bbK^\mathrm{red}$ and use formula \eqref{eq:ReducOnsager} provided by Proposition~\ref{pr:ReducedOnsager}, which takes the specific form 
\[
	\bbK_{\alpha,\beta}\big(\kappa \sfG(\Sigma,m)\big) \ol\xi = \rmD \Uppsi(\Sigma,m,\kappa) (\wh\Sigma,\wh m, \wh \kappa) , \quad \rmD \Uppsi(\Sigma,m,\kappa)^*\xi= (S,\mu,k) \in \rmT^*_\sfp\mfP,
\]
where $ \sfp = (\Sigma,m,\kappa)\in \mfP$ , 
$\ol\lambda=(\wh\Sigma,\wh m,\wh\kappa)$, and $\eta=(S,\mu,k)$. 

We will show that it is possible to find a solution $\ol\xi$ that is a quadratic
polynomial. Hence, it lies in a finite-dimensional set with exactly the same 
dimension as $\rmT^*_\sfp\mfP $, namely $\frac12d(d{+}1) + d +1$.  
Hence, we can reduce the analysis to an
algebraic calculation for finding the coefficients.

More precisely, the equation $\rmD \Uppsi(\sfp)^* \xi= \eta$ means
\[
	\big\langle \xi,\rmD\Uppsi(\sfp) \wh\sfp
    \big\rangle_{\rmT_{\Uppsi(\sfp)}\rmL^2(\R^d)} = \langle \eta, 
    \wh\sfp \rangle_{\rmT_\sfp \mfP} 
\quad \text{for all } \wh \sfp \in \rmT_\sfp\mfP,
\]
and the equation $\bbK_{\alpha,\beta}(\Uppsi(\sfp))\ol\xi = \rmD\Uppsi(\sfp) \ol\lambda$ 
implies that $\ol\xi$ is quadratic, allowing us to use the ansatz
\[
\ol\xi(x) =\xi_{A,b,c}(x):= (x{-}m){\cdot} A(x{-}m) + b{\cdot} (x{-}m) + c.
\]
Using the explicit form of $\rmD\Uppsi(\sfp) \wh \sfp$ with $\wh \sfp = 
(\wh \Sigma, \wh m, \wh\kappa)$, we have to satisfy
\begin{align*}
\int_{\R^d} \xi_{A,b,c}(x)\Big[ \wh\kappa + \kappa \,\wh m\cdot
\Sigma^{-1}(x{-}m)& + \frac\kappa2\big((x{-}m){\cdot} \Sigma^{-1}\wh\Sigma
\Sigma^{-1}(x{-}m) {-} \Sigma^{-1}{\mdot} \wh \Sigma\big) \Big] \sfG(\Sigma,m) \dd x \\
& = S{:}\wh\Sigma + \mu{\cdot}\wh m + k\,\wh\kappa \quad \text{for all } (\wh
\Sigma, \wh m, \wh\kappa) \in \rmT_\sfp\mfP. 
\end{align*} 
Substituting $x=m+\Sigma^{1/2}y$ and using the Gaussian integrals
\eqref{eq:A.DD02} and \eqref{eq:A.DD4}, we obtain 
the following linear system for determining $(A,b,c)$, and hence 
$\ol\xi = \xi_{A,b,c}$: 
\begin{align*}
	\wh\kappa\big( \trace(A\Sigma){+}c\big) + \wh m\cdot(\kappa b) + \kappa \trace(A\wh\Sigma) = S{:}\wh\Sigma + \mu{\cdot}\wh m + k\,\wh\kappa 
\end{align*}
Thus, we have found the unique solution $\ol\xi=\Xi_{S,\mu,k}= \xi_{A,b,c}$ with 
\[
	A= \frac1\kappa S, \qquad 
	b=\frac1\kappa\,\mu,\qquad 
	c= k-\frac{1}{\kappa}\,\trace(S\Sigma). 
\]
Using \eqref{eq:ReducOnsager.a}, the reduced Onsager operator $\bbK^\mathrm{red}$ 
is now defined via
\[
\Big\langle \bma{@{}c@{}}S\\[-0.4em]\mu\\[-0.2em] k \ema\!, 
 \bbK^\mathrm{red}(\sfp)\!
\bma{@{}c@{}}S\\[-0.4em]\mu\\[-0.2em] k \ema \Big\rangle_{\rmT_\sfp \mfP} 
:= \big\langle \Xi_{S,\mu,k} ,
\bbK_{\alpha,\beta}\big(\Uppsi(\sfp)\big) \Xi_{S,\mu,k} 
\big\rangle_{\rmT_{\Uppsi(\sfp)}\rmL^2(\R^d)} .
\]

Since the linear mappings $(S,\mu,k) \mapsto (A,b,c) \mapsto \xi_{A,b,c} $ are completely independent of $\alpha$ and $\beta$, we can construct $\bbK^\mathrm{red}$ by doing the reduction for $\bbK_\Ot$ and $\bbK_{\He}$ independently, and then combining the result.

\begin{lemma}\label{lem:reduced-Onsager}
	For $\sfp=(\Sigma,m,\kappa)\in\mfP$, the reduced Onsager operator $\bbK_{\alpha,\beta}^\mathrm{red}(\sfp)$ is given by
\begin{align}
	\begin{aligned}
	\bbK_{\alpha,\beta}^\mathrm{red}(\sfp) &= \alpha \bbK_\Ot^\mathrm{red}(\sfp) + \beta \bbK_{\He}^\mathrm{red}(\sfp) \\[0.3em]
	&= \alpha \!\bma{ccc}\frac2{\kappa}\big(
\Box \Sigma + \Sigma \Box\big) &0&0 \\ 0 & \frac1\kappa\,\Box & 0\\ 0&0& 0 \ema + \beta \!\bma{ccc}\frac2\kappa\, \Sigma \Box
\Sigma &0&0 \\ 0 & \frac1\kappa\,\Sigma \Box& 0\\ 0&0& \kappa \Box\ema.
	\end{aligned}
\end{align}
\end{lemma}
\begin{proof}
	The calculation is straightforward. Let $\eta = (S,\mu,k)^\top$. Then, for the Otto-Onsager operator, we deduce
\begin{align*}
&\langle \eta, \bbK_\Ot^\text{red}(\sfp)\eta\rangle_{\rmT^*_\sfp\mfP} 
= \big\langle \Xi_{S,\mu,k} ,
\bbK_\Ot\big(\kappa \, \sfG(\Sigma,m)\big) \Xi_{S,\mu,k} 
\big\rangle_{\rmT_{\kappa \sfG(\Sigma,m)}\rmL^2(\R^d)} 
\\
&\qquad= \int_{\R^d}\!\! \kappa \, \sfG(\Sigma,m;x) \big|\nabla \Xi_{S,\mu,k}(x)\big|^2 \dd x
= \kappa \int_{\R^d}\!  \sfG(\Sigma,0;y)\big| 2Ay+b|^2 \dd y 
\\
&\qquad=\kappa \int_{\R^d}\!  \sfG(\Sigma,0;y)\big(4 y{\cdot} A^2y+4Ab{\cdot}y +|b|^2\big) \dd y 
= \kappa \big(4\trace(\Sigma^{1/2}A^2\Sigma^{1/2}) + |b|^2 \big)  
\\
&\qquad= \frac4\kappa \trace(S^2\Sigma) + \frac1\kappa |\mu|^2,
\end{align*}
from which we obtain $\bbK_\Ot^\text{red}(\sfp)$. Similarly, we compute
\begin{align*}
&\langle \eta,\bbK_{\He}^\text{red}(\sfp)\eta\rangle_{\rmT^*_\sfp\mfP} 
= \big\langle \Xi_{S,\mu,k} ,
\bbK_{\He}\big(\kappa \sfG(\Sigma,m)\big) \Xi_{S,\mu,k} 
\big\rangle_{\rmT_{\kappa \sfG(\Sigma,m)}\rmL^2(\R^d)} 
\\
&\qquad= \int_{\R^d} \!\!\kappa \sfG(\Sigma,m;x) \big(\Xi_{S,\mu,k}(x)\big)^2 \dd x
= \kappa\int_{\R^d}  \! \sfG(\Sigma,0;y) \big(y{\cdot}Ay + b{\cdot} y +c\big)^2 \dd x
\\
&\qquad= \kappa \Big(2\big|\Sigma^{1/2}A\Sigma^{1/2}\big|^2 
     + \big|\Sigma^{1/2}b\big|^2 +\big(\trace(A\Sigma)+
c\big)^2  \Big)  
\\
&\qquad= \frac2\kappa |\Sigma^{1/2}S\Sigma^{1/2}|^2 + \frac1\kappa
|\Sigma^{1/2}\mu|^2 + \kappa k^2,\qquad \eta = (S,\mu,k)^\top,
\end{align*}
where we used the matrix norm notation $|\cdot|^2:=\|\cdot\|_F^2$ for the Frobenius norm. This gives the reduced Hellinger-Onsager operator $\bbK_{\He}^\text{red}$ as asserted.
\end{proof}

A remarkable fact is that $\bbK^\mathrm{red}_{\alpha,\beta}$ is again of the additive form with coefficients $\alpha$ and $\beta$. For general Schur complements as in Remark \ref{re:Schur}, it cannot be expected that $\bbK^\mathrm{red}$ retains an additive structure. 

\begin{remark}\label{rem:}
From the reduced Onsager operators $\bbK_\Ot^\textrm{red}$ and $\bbK_\He^\textrm{red}$, one obtains dynamical transport costs taking the following form for $j\in\{\Ot,\He\}$:
\[
	\mfd_j^2(\sfp_0,\sfp_1) = \inf_{(\sfp,\eta)}\left\{\int_0^1 \langle\eta, \bbK_j^\mathrm{red}(\sfp)\eta\rangle_{\rmT^*_\sfp\mfP}\,\dd s :\; \dot\sfp = \bbK_j^\mathrm{red}(\sfp)\eta,\;\;\sfp(0)=\sfp_0,\;\sfp(1)=\sfp_1\right\}.
\]
Those represent the dynamical counterparts of the 2-Wasserstein and Hellinger distances restricted to the parameter space $\mfP$, respectively.
\end{remark}
\EEE

\subsection{Reduced gradient system} 

Based on the Lemma~\ref{lem:reduced-Onsager}, we are now left to derive the reduced driving energy, which we denote by $\sfE$. To do so, we simply evaluate the relative Boltzmann entropy in terms of the parametrization variables $\sfp=(\Sigma,m,\kappa)\in \mfP$ to obtain
\begin{equation}
\label{eq:relEntrRed}
\sfE(\sfp) = \calH_\rmB\big(\kappa \sfG(\Sigma,m) \big|
\varkappa \sfG(\Gamma,n)\big) = \kappa\, \sfH(\Sigma,m|\Gamma,n) + \LB\left(\frac\kappa\varkappa \right) \varkappa,
\end{equation}
with
\[
	\sfH(\Sigma,m|\Gamma,n)= \frac12(m{-}n)\cdot \Gamma^{-1}(m{-}n) +\frac12 \trace\big(\Gamma^{-1}\Sigma - I\big) - \frac12 \log\left( \frac{\det\Sigma}{\det\Gamma} \right).
\]
Note that $\sfH$ is simply the relative Boltzmann entropy
$\calH_\rmB\big( \sfG(\Sigma,m) \big|\sfG(\Gamma,n)\big)$ between two normalized Gaussians, and hence $\sfH(\Sigma,m|\Gamma,n)\geq 0$ with equality if and only if $\Sigma = \Gamma$ and $m=n$ (see also \cite[Eqn.\,(F59)]{FloSch24REMI}). Moreover, due to the strict concavity of the map $\bbR_{>0}^{d\ti d}\ni \Sigma\mapsto \log \det \Sigma$, we find that
	\[
		(\Sigma,m)\mapsto \sfH(\Sigma,m|\Gamma, n)\qquad\text{is strictly (jointly) convex}.
	\]
Its differential is given by
\begin{equation}
  \label{eq:rmDsfE}
  \rmD \sfE(\sfp) = \bma{c} \rmD_\Sigma \sfE(\Sigma, m,\kappa) \\
 \rmD_m \sfE(\Sigma, m,\kappa) \\
 \rmD_\kappa \sfE(\Sigma, m,\kappa) \ema
 = \bma{c} \frac\kappa2 \big( \Gamma^{-1} - \Sigma^{-1}\big) \\
\kappa\, \Gamma^{-1} (m{-}n) \\
\sfH(\Sigma,m|\Gamma,n) + \log(\kappa/\varkappa)
\ema.
\end{equation}

This leads to the following main result concerning the reduced gradient structure. 

\begin{theorem}
 \label{th:ReducedGradStr}
The system of ODEs \eqref{eq:ODE.mfP} is the gradient-flow equation of the 
gradient system $(\mfP,\sfE,\bbK^\mathrm{red}_{\alpha,\beta})$. 
\end{theorem}

The proof can be done in two ways: We can insert the given formulas for $\bbK^\mathrm{red}_{\alpha,\beta}(\sfp)$ and for $\rmD \sfE(\sfp)$ and 
observe that \eqref{eq:ODE.mfP} has the form 
$\dot \sfp= -\bbK^\mathrm{red}_{\alpha,\beta}(\sfp)\rmD\sfE(\sfp)$, which reads
\begin{subequations}
\label{eq:ODE22.mfP}
\begin{align}
\dot \Sigma&= \alpha \big(2 I - \Gamma^{-1}\Sigma - \Sigma \Gamma^{-1} \big) 
 + \beta\big( \Sigma -\Sigma \Gamma^{-1} \Sigma \big),\label{eq:ODE22.mfP.sigma}
\\[0.3em]
\dot m &= -\bigl(\alpha +\beta\Sigma\bigr) \Gamma^{-1} (m {-} n ),
\label{eq:ODE22.mfP.m}\\
\dot \kappa &= - \beta \kappa\big(\log(\kappa/\varkappa) + \sfH(\Sigma,m|\Gamma,n)\big) \label{eq:ODE22.mfP.kappa}. 
\end{align}
\end{subequations}
Alternatively, it is a consequence of the general Theorem~\ref{th:ReductOnsag} applied to our specific case. 

In Theorem \ref{thm:HK.Gauss.mass-shape-ode} below, we will establish the system of ODEs for cases where the target measure may not be Gaussian.
\begin{remark}\label{rem:explicit.kappa}
	The solution $\kappa$ to \eqref{eq:ODE22.mfP.kappa} can be explicitly computed to obtain
	\[
		\kappa(t) = \varkappa \left(\frac{\kappa_0}{\varkappa}\right)^{\ee^{-\beta t}}\exp\left(-\int_0^t \ee^{-\beta(t-s)} \sfH(\Sigma(s),m(s)|\Gamma,n)\dd s\right).
	\]
\end{remark}

In view of Remark~\ref{rem:explicit.kappa}, it suffices to study the system of equations \eqref{eq:ODE22.mfP.sigma}--\eqref{eq:ODE22.mfP.m} for $(\Sigma,m)$. It turns out that this system can be derived as a reduced gradient flow from the evolution of normalized Gaussians in $\calP(\bbR^d)$ as described in the following.

\subsection{Reduction to normalized Gaussians}
\label{se:GaussProba}
In the preceding sections, we analyzed gradient flows in the space of Gaussian measures with variable mass. For normalized Gaussians in $\mfG_1=\Upphi(\mfP_1)\subset\calP(\bbR^d)$, we consider the gradient structure $(\calP(\R^d),\calE , \SHK_{\alpha,\beta})$ instead, where $\SHK_{\alpha,\beta}$ is the \emph{spherical} Hellinger-Kantorovich distance induced by the Onsager operator 
\[
	\bbK^\SHK_{\alpha,\beta}(\rho)\xi = \alpha\bbK_{\Ot} (\rho)\xi + \beta \bbK_{\SHe} (\rho)\xi = - \alpha \DIV\!\big(\rho\nabla \xi\big) + \beta \rho\Big( \xi -{\ts \int_{\R^d} \rho \,\xi\dd x}\Big).
\]
The associated gradient-flow equation then reads
\begin{align}
\label{eq:SHK.PDE}
\dot \rho &= \alpha \DIV\!\big( \nabla \rho + \rho\nabla \log \target\big) 
 - \beta \rho\left( \!\log\left(\frac{\rho}\target\right)
 -\int_{\R^d}\rho\log\left(\frac{\rho}\target\right) \rmd x  \right). 
\end{align}
In \cite{mielke2025hellinger}, it was observed that the solutions to the $\HK_{\alpha,\beta}$-flow \eqref{eq:I.HK.PDE} and the $\SHK_{\alpha,\beta}$-flow \eqref{eq:SHK.PDE} are closely related as given in the following statement.

\begin{lemma}\label{lem:HK.SHK.shape-mass}
 If $t \mapsto \uu(t)$ solves \eqref{eq:I.HK.PDE}, then $t\mapsto \rho(t)= \uu(t)/z(t)$ with $z(t)=\int_{\R^d} \uu(t,x)\dd x$ solves \eqref{eq:SHK.PDE}. Conversely, if $t\mapsto \rho(t)$ solves \eqref{eq:SHK.PDE}, then $t \mapsto \uu(t)=\kappa(t) \rho(t)$ solves \eqref{eq:I.HK.PDE} for a suitable function $t\mapsto \kappa(t)$.
\end{lemma}
This observation allows us to separately analyze the evolution of the shape (normalized measure $\rho$) and the mass ($z$) of the solution to $\HK$ flow \eqref{eq:I.HK.PDE}, where the shape evolution is governed by the $\SHK$ flow and does not depend on the mass.

Since the manifold $\mfG_1\subset\calP(\bbR^d)$ is invariant under the $\SHK_{\alpha,\beta}$-flow whenever $\target\in\mfG_1$, a gradient structure on $\mfP_1=\R^{d\ti d}_\spd \ti \R^d$ can be derived from $(\calP(\R^d),\calE , \SHK_{\alpha,\beta})$ using the same approach as before to obtain the following result.

\begin{theorem}[Gaussian $\SHK_{\alpha,\beta}$-GFE]\label{thm:reduced-SHK}
The reduced gradient system on $\mfP_1$ takes the form $(\mfP_1,\sfE_1,\bbK_{\SHK}^\mathrm{red})$, where $\sfE_1(\sfp)= \sfH(\sfp| \Gamma,n)$, $\sfp=(\Sigma,m)$, and
\[
	\bbK_{\SHK}^\mathrm{red}(\sfp) = \alpha\bbK_{\Ot}^\mathrm{red} (\sfp) + \beta \bbK_{\SHe}^\mathrm{red} (\sfp) = \alpha  \bma{cc} \!\!2\big( \Box \Sigma {+} \Sigma \Box\big)\!\! &0 \\ 0 & \!1\,\Box\!  \ema + \beta \bma{cc}\!\!2\, \Sigma \Box \Sigma \!\!&0 \\ 0 & \Sigma \Box\!\ema.
\]
The associated equations are obtained from
\eqref{eq:ODE22.mfP} by dropping the last equation:
\begin{subequations}
\label{eq:ODE66.mfP}
\begin{align}
\dot \Sigma&= \alpha \big(2 I - \Gamma^{-1}\Sigma - \Sigma \Gamma^{-1} \big) 
 + \beta\big( \Sigma - \Sigma \Gamma^{-1} \Sigma \big),
\\[0.3em]
\dot m&=  \alpha \, \Gamma^{-1} (n {-} m ) +
\beta \,\Sigma \Gamma^{-1}(n{-} m),
\end{align}
\end{subequations}
\end{theorem}

\subsection{Extension to non-Gaussian target measure $\target$}
\label{se:ExtenNonGauss}

In the previous sections, we have considered the driving energy $\calH_\rmB(\,\cdot\,|\target)$ for a Gaussian target measure $\target$. We now consider the general case when the target measure $\target$ is not necessarily Gaussian,
which is the case in various applications.
We first show a simple formal calculation to provide the ODE for the $\HK_{\alpha,\beta}$-flow
for a general target measure $\target\in \calM(\bbR^d)$.
Note that the formal gradient-flow equations have been derived for the pure Wasserstein case in \cite{lambertVariationalInferenceWasserstein2022} and for the pure Fisher-Rao (i.e., spherical Hellinger) case in \cite{chenGradientFlowsSampling2023}.
Using our reduced Onsager structure on the parameter space, we can easily write down the ODE for the reduced Gaussian $\HK_{\alpha,\beta}$-flow.

For a general target measure $\target\in \calM(\bbR^d)$, the reduced relative Boltzmann entropy takes the following form: For $\sfp=(\Sigma,m,\kappa)\in\mfP$,
\begin{align*}
	\sfE(\sfp) &= \calH_\rmB\big(\kappa \sfG(\Sigma,m) \big|\target\big) = |\target|\lambda_\rmB\Big(\frac{\kappa}{|\target|}\Big) + \kappa \,\calH_\rmB\Big(\sfG(\Sigma,m)\,\Big|\,\frac{\target}{|\target|}\Big).
\end{align*}
For target measures $\target$ with Lebesgue density, the differential $\rmD\sfE$ can be computed using Remark~\ref{rem:psi-derivative} to obtain
\begin{align*}
	\rmD_\Sigma\sfE(\Sigma,m,\kappa) &= -\frac{\kappa}{2}\left(\Sigma^{-1} + \int_{\R^d} \nabla_x^2 \log \target\,\,\sfG(\Sigma,m)\dd x \right), \\
	\rmD_m\sfE(\Sigma,m,\kappa) &= -\kappa \int_{\R^d} \nabla_x\log \target\,\, \sfG(\Sigma,m)\dd x, \\
	\rmD_\kappa\sfE(\Sigma,m,\kappa) &= 
	\log \Big(\frac{\kappa}{|\target|}\Big) + \calH_\rmB \Big( 
    \sfG(\Sigma,m)\,\Big|\,\frac{\target}{|\target|}\Big).
\end{align*}

Compared with the differentials for the Gaussian target measure in \eqref{eq:rmDsfE},
the differentials for the general target measure can be obtained by replacing the quantity $\left(\Gamma^{-1}, \Gamma^{-1}(m-n)\right)$ with the quantity
$\left(- \int \nabla_x^2 \log \target\,\,\sfG(\Sigma,m)\dd x,  - \int \nabla_x\log \target\,\, \sfG(\Sigma,m)\dd x\right)$.
For ease of notation, we set
\[
	\Gamma^{-1}=\Gamma^{-1}(\Sigma,m) := -\int_{\R^d} \nabla_x^2 \log \target
    \;\sfG(\Sigma,m)\dd x.
\]

The evolution for $\sfp=(\Sigma,m,\kappa)$ can be deduced from $\dot \sfp= -\bbK^\mathrm{red}_{\alpha,\beta}(\sfp)\rmD\sfE(\sfp)$, yielding

\begin{theorem}
  [$\HK$-Gaussian shape-mass ODE]
  \label{thm:HK.Gauss.mass-shape-ode}
  Let $\target\in\calM(\bbR^d)$ be a target measure with Lebesgue density. Then the ODE for the reduced Gaussian
  $\HK_{\alpha,\beta}$-gradient flow on $\sfp=(\Sigma,m,\kappa)\in \mfP$ with driving energy $\sfE$
  is given by
\begin{subequations}\label{eq:general-target}
\begin{align}
	\dot\Sigma &= \alpha \bigl(2I - \Gamma^{-1}\Sigma -\Sigma \Gamma^{-1}\bigr) 
    + \beta\bigl(\Sigma - \Sigma \Gamma^{-1}\Sigma\bigr), 
    \label{eq:HK.Gauss.mass-shape-ode-covariance}
\\
	\dot m &= \int_{\R^d} (\alpha {+} \beta\Sigma)\: \nabla_x\log\target 
    \; \sfG(\Sigma,m)\dd x, \label{eq:HK.Gauss.mass-shape-ode-mean}
\\
	\dot\kappa &= 
	-\beta \kappa \left(\log \Big(\frac{\kappa}{|\target|}\Big) + \calH_\rmB\Big(\sfG(\Sigma,m)\,\Big|\,\frac{\target}{|\target|}\Big)\right).
	\label{eq:HK.mass-ode}
\end{align}
\label{eq:HK.Gauss.mass-shape-ode-all}
\end{subequations}

Furthermore, the mass ODE \eqref{eq:HK.mass-ode} admits the explicit solution 
(cf.\ Remark~\ref{rem:explicit.kappa})
\[
	\kappa(t) = |\target| \Big(\frac{\kappa_0}{|\target|} 
    \Big)^{\ee^{-\beta t}} \!\!\exp\left(  - \int_0^t 
    \ee^{-\beta (t-s)} \: \calH_\rmB\Big(\sfG(\Sigma(s),m(s))\,\Big|\,
    \frac{\target}{|\target|}\Big)\rmd s\right).
\]
\end{theorem}
Similarly, for the $\SHK_{\alpha,\beta}$-flow one obtains the reduced coupled system \eqref{eq:HK.Gauss.mass-shape-ode-covariance}--\eqref{eq:HK.Gauss.mass-shape-ode-mean} on $\mfP_1$ with the driving energy $\sfE_1(\sfp) = \calH_\rmB(\sfG(\Sigma,m)|\target)$.
The gradient system \eqref{eq:general-target} with $\alpha=0$ or $\beta=0$ have appeared in the context of Kalman filtering and state estimation, variational Gaussian approximations and information geometry (cf.\ \cite{opper-archambeau2009,Pidstrigach-Reich2023,sarkka2007} and references therein).

\section{Geodesic convexity of the reduced gradient system}
\label{se:GeodesicConvexity}

Geodesic convexity is a powerful geometric tool, provided it is applicable. For the pure transport case, i.e., the Otto-Wasserstein geometry, where it is also called displacement convexity according to \cite{Mcca97CPIG}, there is a large amount of literature (see, e.g., \cite{OttVil00GITL,Vill03TOT,AmGiSa05GFMS}) concerning the related functional inequalities. The corresponding theory for the Hellinger-Kantorovich geometry was only recently developed in \cite{LiMiSa23FPGG}. 

Here, we investigate how the concept of geodesic $\uplambda$-convexity applies in the reduced setting of the Gaussian family and the relative Boltzmann entropy with a Gaussian target measure $\target$.
We will restrict to the case of normalized Gaussians $\mfG_1$ because in the case with varying mass, global geodesic $\uplambda$-convexity cannot be expected, see Remark \ref{rem:GLCvx.ScaGauss}.

\subsection{Geodesic convexity and the metric Hessian}
\label{se:GeodCvxHessian}

Working with the finite-dimensional gradient system $(\mfP,\sfE,\bbK^\mathrm{red}_{\alpha,\beta})$, we can define the metric Hessian without the technical difficulties arising 
in the  full Hellinger-Kantorovich setting for $(\calM(\R^d),\calE, \HK_{\alpha,\beta})$ 
or $(\calP(\R^d),\calE, \SHK_{\alpha,\beta})$. 

For any smooth finite-dimensional gradient system $(\mfQ,\sfF,\bbK)$, the metric Hessian (in its co-variant form) is a symmetric tensor $\HESS_{\bbK}\sfF(\sfq): \rmT^*_\sfq\mfQ \to \rmT_\sfq \mfQ$ defined in terms of the quadratic form  
\begin{equation}
    \label{eq:Def.Hessian}
\begin{aligned}
    \big\langle \xi, \HESS_{\bbK} \sfF(\sfq) \xi \big\rangle
&:= \langle \xi,\bbK(\sfq)\rmD^2\sfF(\sfq) \bbK(\sfq)\xi\rangle \\
 &\quad + \langle \xi,\rmD \bbK(\sfq)[\bbK(\sfq)\xi]\rmD \sfF(\sfq)\rangle - \frac12 \langle \xi, \rmD \bbK(\sfq)[\bbK(\sfq)\rmD\sfF(\sfq)] \xi\rangle,
\end{aligned}
\end{equation}
where $\rmD\bbK(\sfq)[v]$ is the directional derivative $\lim_{s\to 0} \frac1s 
\big(\bbK(\sfq{+}sv){-}\bbK(\sfq)\big)$ (in local coordinates). 
Often it is advantageous to define the vector field $\bfV(\sfq)=\bbK(\sfq)\rmD \sfF(\sfq)$, then one has the slightly simpler formula 
\begin{equation}
    \label{eq:HessianSimpler}
 \big\langle \xi, \HESS_{\bbK} \sfF(\sfq) \xi \big\rangle
 = \langle \xi,\rmD \bfV(\sfq)[ \bbK(\sfq)\xi] \rangle 
  - \frac12 \langle \xi, \rmD \bbK(\sfq)[\bfV(\sfq)] \xi\rangle, 
\end{equation}

Its relation to convexity along geodesics is immediate; we refer to \cite{OttVil00GITL,OttWes05ECCW,DanSav08ECDC,LieMie13GSGC} for details. For this, we recall the geodesic equations for constant-speed geodesics $s\mapsto \sfg_s \in \mfQ$: 
\begin{equation}
    \label{eq:GeodEqns}
  \dot \sfg_s = \bbK(\sfg_s) \xi_s, \qquad   \dot  \xi_s = - \frac12\langle \xi_s, \rmD \bbK(\sfg_s)[\, \cdot \,] \xi_s \rangle, 
\end{equation}
which are Hamilton's equation for the Hamiltonian $\sfH(\sfq,\xi)=\frac{1}{2}\langle \xi,\bbK(\sfq)\xi\rangle$. 

\begin{remark}
Although we will not use them explicitly, we state the equations for the 
constant-speed geodesics, considering the general case of scaled Gaussians
parameterized by the set $\mfP =[0,\infty)\mfP_1$. Using $\sfp_s = 
(\Sigma_s,m_s,\kappa_s)$ and $\xi_s=(S_s,\mu_s,k_s)$ with $s\in [0,1]$ 
being the arclength parameter, we find 
\begin{subequations}
    \label{eq:GaussGeodEqn}
\begin{align}
\dot \Sigma_s &= \frac{2\alpha}{\kappa_s}(S_s\Sigma_s {+} \Sigma_s S_s) + \frac{2\beta}{\kappa_s}\Sigma_s S_s\Sigma_s , \\
\dot m_s &=\frac{\alpha}{\kappa_s}\mu_s + \frac{\beta}{\kappa_s}\Sigma_s\mu_s, \hspace{13em}
 \dot \kappa_s=0+\beta \kappa_s k_s ,
\\
\dot S_s&=\frac{2\alpha}{\kappa}_s S_sS_s -\beta\Big(\frac{2}{\kappa_s}S\Sigma_s S_s +\frac{1}{2\kappa_s}\mu_s\otimes\mu_s\Big), \hspace{3em}
\dot \mu_s =0, \\
\dot k_s & = -\frac{\alpha}{\kappa_s^2}\Big(2\trace[S_s^2\Sigma_s] +\frac{1}{2}|\mu_s|^2\Big)
+\frac{\beta}{2\kappa_s^2}\Big(\mu_s{\cdot}\Sigma_s\mu_s + 2\trace[S_s\Sigma_s S_s\Sigma_s]\Big) +\frac{\beta}{2} k_s^2.
\end{align}
\end{subequations}
The equations on $\mfP_1$ are obtained by setting $\kappa_s\equiv 1$ and dropping the equations for $k_s$. 
\end{remark}

The global $\uplambda$-convexity of $\calE $ in the space $(\mfQ,\bbK)$ means that for all (constant-speed) geodesics the function $s \mapsto \sfF(\sfg_s)$ is $\uplambda$-convex, i.e., its second derivative is bounded from below by $\uplambda$. However, we compute 
\[
\frac{\dd }{\dd s}\sfF(\sfg_s) 
=\langle \rmD\sfF(\sfg_s), \dot \sfg_s\rangle =\langle \rmD\sfF(\sfg_s), \bbK(\sfg_s)\xi_s \rangle = \langle \xi_s, \bfV(\sfg_s)\rangle . 
\]
Thus, taking another derivative with respect to $s$ gives 
\[
\frac{\dd^2}{\dd s^2} \sfF(\sfg_s) = \langle \xi_s,\rmD \bfV(\sfg_s)[\dot \sfg_s]\rangle + \langle \dot\xi_s,\bfV(\sfg_s)\rangle = \langle \xi_s , \HESS_\bbK \sfF(\sfg_s) \,\xi_s\rangle. 
\]
Thus, \emph{(global) geodesic $\uplambda$-convexity} reduces to the condition 
\begin{equation}
    \label{eq:Def.GLCvx}
    \forall\, \sfq\in \mfQ,\, \forall \,\xi \in \rmT_\sfq\mfQ:\quad 
    \langle \xi, \HESS_\bbK\sfF(\sfq)\,\xi\rangle \geq \uplambda \langle \xi, \bbK(\sfq)\xi\rangle. 
\end{equation}
The last estimate can also be written in short as $\HESS_\bbK\sfF(\sfq) \geq \uplambda\, \bbK(\sfq)$, which is meant in the sense of quadratic forms. 

We will also use the notion of \emph{sublevel geodesic semi-convexity}, which is defined as follows:
\begin{equation}
    \label{eq:SublGLCvx}
\begin{aligned}
    & \forall \, E\in \R,\, \exists\, \wt\uplambda(E)\in \R,\, 
    \forall \, q\in \mfQ \text{ with }\sfF(\sfq)\leq E 
    :\quad   
   \HESS_\bbK\sfF(\sfq) \geq \wt\uplambda(E) \bbK(\sfq). 
\end{aligned}
\end{equation}
The latter condition can be used to control the expansion/contraction properties 
between two solutions of the gradient-flow equations starting in the same sublevel, 
see \cite{MiSaTs25?SEVI}.   

\subsection{The Hessian for the Gaussian family } 
\label{su:HessGauss}

We now return to the normalized Gaussian gradient system $(\mfP_1,\sfE_1,
\bbK^\mathrm{red}_{\alpha,\beta})$ with driving energy $\sfE_1(\sfp)=\sfH(\sfp|\Gamma,n)$, $\sfp=(\Sigma,m)$, and introduce the shorthand  
\[
	\bbH_{\alpha,\beta}(\sfp)= \HESS_{\bbK^\mathrm{red}_{\alpha,\beta}} \sfE_1 (\sfp). 
\]
From the definition \eqref{eq:Def.Hessian}, we know that $\bbH_{\alpha,\beta}$ is linear in $\sfE_1$ and quadratic in $ \bbK^\mathrm{red}_{\alpha,\beta}$. Hence, we can write 
\[
 \bbH_{\alpha,\beta}= \alpha^2 \bbH_\Ot + \alpha\beta \bbH_\mafo{mix} 
  + \beta^2 \bbH_\SHe. 
\]

For the explicit calculation, we recall from Theorem~\ref{thm:reduced-SHK} the reduced operators
\[
\bbK^\mathrm{red}_{\alpha,\beta}(\sfp)= \alpha \bbK^\mathrm{red}_\Ot(\sfp) + \beta
\bbK^\mathrm{red}_{\SHe}(\sfp) 
= \alpha   \bma{@{}ccc@{}} 2\big(
\Box \Sigma + \Sigma \Box\big) &0 \\ 0 & \Box \ema
+ \beta  \bma{@{}ccc@{}} 2\Sigma \Box \Sigma &0 \\ 
0 & \Sigma \Box \ema ,
\]
The driving forces $\rmD\sfE_1$ and the vector field $\bfV= \bbK^\mathrm{red}_{\alpha,\beta} \rmD\sfE_1 $ take the form 
\[
  \rmD \sfE_1(\sfp) 
 = \binom{ 2 ( \Gamma^{-1} {-} \Sigma^{-1}) } 
    { \Gamma^{-1} (m{-}n)} \quad\text{and}\quad 
    \bfV(\sfp)= \alpha
 \binom{\Gamma^{-1}\Sigma{+}\Sigma \Gamma^{-1}{-}2I}{\Gamma^{-1} (m{-}n) } 
  + \beta \binom{\Sigma\Gamma^{-1}\Sigma{-}\Sigma}{\Sigma\Gamma^{-1} (m{-}n) } .
\]
A direct calculation using formula \eqref{eq:HessianSimpler} leads to the 
following results:
\begin{subequations}
 \label{eq:ExplicHessians}
\begin{align}
    \label{eq:ExplicHessian.a}
    \langle \eta, \bbH_\Ot(\sfp)\eta\rangle
& = 4\trace (S\Gamma^{-1}S\Sigma) +4|S|^2 + \mu\cdot\Gamma^{-1}\mu .
\\
  \nonumber
  \langle \eta, \bbH_\mafo{mix}(\sfp)\eta\rangle
  &= 2 \trace\Big( 2\Gamma^{-1}(\Sigma S \Sigma S{+} S\Sigma S \Sigma {+} \Sigma S S \Sigma\big) - S\Sigma S \Sigma  \Big)
 \\   
  \label{eq:ExplicHessian.b}
    &\quad + 2 \mu\cdot(\Sigma S{+}S\Sigma) \Gamma^{-1}(m{-}n) + \mu\cdot( \Sigma\Gamma^{-1}{+}I) \mu,
\\
\nonumber
\langle \eta, \bbH_\SHe(\sfp)\eta\rangle
& =2\trace (S\Sigma S\Sigma\Gamma^{-1}\Sigma ) 
+ 2\mu\cdot\Sigma S\Sigma \Gamma^{-1}(m{-} n)
\\
&   \label{eq:ExplicHessian.c}
 \quad +\frac{1}{2}\mu\cdot(\Sigma\Gamma^{-1}\Sigma{+}\Sigma)\mu,\qquad \eta=(S,\mu)^\top . 
\end{align}
\end{subequations}

With these results, we can now make our first statement on geodesic 
$\uplambda$-convexity. For $\beta>0$ there is no $\uplambda$-convexity, whereas 
for $\beta=0$ we have the well-known result 
that $\uplambda= \alpha\nu_\mafo{min}(\Gamma^{-1})$, cf.\ \cite{bakryDiffusionsHypercontractives1985}.

\begin{proposition}[Geodesic $\uplambda$-convexity for {$(\mfP_1,\sfE_1, 
\bbK^\mathrm{red}_{\alpha,\beta}) $}]
\label{prop:mfP1.non.GLC}
For $(\mfP_1,\sfE_1,\bbK^\mathrm{red}_{\alpha,\beta})$ we have geodesic 
$\uplambda$-convexity if and only if
\[
  \beta=0 \quad \text{and} \quad \uplambda \leq \alpha \nu_\mafo{min}(\Gamma^{-1}) .
\]
\end{proposition}
\begin{proof}
We first consider the case $\beta>0$, then we observe that $\bbH_{\alpha,\beta}$ 
contains the term 
\begin{equation}
    \label{eq:DisturbTerm}
  2\mu \cdot \big( \alpha\beta(\Sigma S{+}S\Sigma) + \beta^2 \Sigma S\Sigma\big) 
  \Gamma^{-1} (m{-}n) . 
\end{equation}
whereas $m$ does not appear elsewhere in $\bbH_{\alpha,\beta}$ or 
$ \bbK^\mathrm{red}_{\alpha,\beta}$. Thus, keeping $\Sigma$, $S$, and $\mu$ fixed, the above term, and hence the left-hand side in \eqref{eq:Def.GLCvx}, can be made arbitrarily small by choosing $m$ suitably. Hence, there is no geodesic 
$\uplambda$-convexity.

In the case $\beta=0$, we are left with the estimate 
\[
\alpha^2 \big( 4\trace(\Gamma^{-1}S\Sigma S)) + 4 |S|^2
 + \mu{\cdot} \Gamma^{-1} \mu\big) 
\geq \alpha \uplambda \big( 4 \trace(S\Sigma S) + |\mu|^2\big)  ,
\]
and the result follows by standard estimates. 
\end{proof}

\begin{remark}[Centered Gaussians]
    \label{rem:GLC.centeredGauss}
A more positive result can be obtained by restricting the analysis to centered Gaussians, namely 
\[
\mfP_{1,\mafo{cent}} = \big\{ \sfp=(\Sigma,m)\in \mfP_1 \:\big| \: m=0\,\big\}. 
\]
With the target measure $\target=G(\Gamma,0)$, the disturbing term in \eqref{eq:DisturbTerm} disappears and it should be possible to show that there is global geodesic $\uplambda$-convexity for some $\uplambda = \uplambda_\mafo{cent}(\alpha,\beta,\Gamma) \in \R$.

Indeed, for $\beta=0$ we have $\uplambda_\mafo{cent}>0 $ by Proposition \ref{prop:mfP1.non.GLC}. For $\alpha=0$ one easily obtains $\uplambda=0$, but in the case $\alpha \beta>0$ one expects $\uplambda_\mafo{cent}(\alpha,\beta,\Gamma) <0 $.
\end{remark}

\subsection{Sublevel geodesic semi-convexity and scaled Gaussians}
\label{su:SublevelGLC}

As shown above, we have a Hessian $\bbH_{\alpha,\beta}$ that is well-defined on all of $\mfP_1$. To establish sublevel geodesic semi-convexity, it is hence sufficient to show that the sublevels define suitably compact sets bounding $\bbH_{\alpha,\beta}$ from below and $\bbK^\mathrm{red}_{\alpha,\beta}$ from above. 

\begin{lemma}[Sublevels of $\sfE_1$]
\label{lem:Sublevel} 
For fixed $(\Gamma,n)\in \mfP_1$,
\begin{equation}
    \label{eq:Sublevel}
  \forall E>0,\, \exists\, r_E>0,\, \forall\,\sfp=(\Sigma,m) \text{ with }\sfE_1(\sfp)\leq E
  : \ \ |m{-}n|,\, |\Sigma|,\, |\Sigma^{-1}| \leq r_E. 
\end{equation}
\end{lemma}
\begin{proof}
The additive structure of $\sfE_1$ implies that $\frac12(m{-}n)\cdot \Gamma^{-1}(m{-}n)\leq E$ which gives $|m{-}n| \leq (2\nu_\mafo{max}(\Gamma)E)^{1/2}=: r_1$.

Setting $\sfF(B)= \trace(B{-}I) - \log\det B\geq0$ for symmetric matrices $B$, we find $\sfF(B) = \sum_{i=1}^d ( b_j {-} 1 {-} \log b_j)$, where $(b_j)_j$ are the eigenvalues of $B$. Using $B= \Gamma^{-1/2} \Sigma \Gamma^{-1/2}$ and $\sfE_1(\Sigma,m)\leq E$, we have  $b_j{-}1{-}\log b_j\leq 2E$ for all $j$, implying the estimate $\ee^{-(1+2E)} \leq b_j \leq  1+ \log 4 + 4E$. Thus, we conclude 
\[
	|\Sigma| \leq |\Gamma| (1+ \log 4 + 4E)=:r_2 \quad \text{and} \quad 
	|\Sigma^{-1}| \leq |\Gamma^{-1}| \ee^{1+2E}=:r_3,
\]
and the assertion is proved with $r_E:=\max\{r_1,r_2,r_3\}$.  
\end{proof}
Note that, in the proof, the use of the auxiliary matrix $B= \Gamma^{-1/2} 
\Sigma \Gamma^{-1/2}$ can also be viewed through the lens of a coordinate
transformation: $x\mapsto \Gamma^{-1/2} (x - n).$
Additionally, define the notation $\widetilde m := \Gamma^{-1/2} (m - n)$.
Under this transformation, we can simplify the relative Boltzmann entropy to 
have the standard Gaussian target measure
$$
\sfE_1(\sfp)= \sfH(\Sigma,m|\Gamma,n)
= \sfH( B, \widetilde m | I, 0)
= \frac12 \left(
  |\widetilde m|^2 + \underbrace{\trace(B - I) - \log {|B|}}_{\sfF(B)}\right).
$$

With this result, we can now establish the desired sublevel geodesic semi-convexity. 

\begin{proposition}[Sublevel geodesic semi-convexity]
\label{prop:SublevelGLC} The gradient system $ \big( \mfP_1,\sfE_1, 
\bbK^\mathrm{red}_{\alpha,\beta}\big) $ is sublevel geodesically semi-convex 
as defined  in \eqref{eq:SublGLCvx}.
\end{proposition}
\begin{proof}
We fix $E>0$ and restrict to $\sfp=(\Sigma,m)$ with $\sfE_1(\sfp)\leq E$. Then, Lemma \ref{lem:Sublevel} provides a lower bound $\Sigma \geq r_E^{-1} I$. With this, we obtain the lower bound 
\[
\langle \eta,\bbK^\mathrm{red}_{\alpha,\beta}(\sfp)\eta\rangle
\geq c_E \big(|S|^2+ |\mu|^2\big) 
\quad \text{ for all $\eta=(S,\mu)^\top$} 
\]
for some small $c_E>0$. Using the upper bounds for $\Sigma$ and $(m{-}n)$ from Lemma \ref{lem:Sublevel}, we also find the lower bound
\[
\langle \eta,\bbH_{\alpha,\beta}(\sfp)\eta\rangle
\geq -H_E \big(|S|^2+ |\mu|^2\big)
\quad \text{ for all $\eta=(S,\mu)^\top$},
\]
for some large $H_E\geq 0$. This implies 
\[
\bbH_{\alpha,\beta}(\sfp)
\geq \uplambda\,\bbK^\mathrm{red}_{\alpha,\beta}(\sfp),\quad \text{with\; $\uplambda=- H_E/c_E$},
\]
which is the desired assertion.
\end{proof}

\begin{remark}[No geodesic $\uplambda$-convexity for scaled Gaussians]
    \label{rem:GLCvx.ScaGauss}
If we consider scaled Gaussians parametrized by $\mfP$ instead of normalized 
Gaussians in $\mfP_1$, we can use the well-known fact that the geodesic curves 
in $(\calM(\R^d), \HK_{\alpha,\beta}) $ connecting any measure $\mu=\mu_1$ with 
the trivial measure $\mu_0=0$ is given by the Hellinger geodesic 
$s\mapsto \mu_s = s^2 \mu_1$ (cf.\ \cite{LiMiSa16OTCR}). 
If $\mu_1$ is a (scaled) Gaussian measure, then so are all $\mu_s$. Hence, 
in the parameter space $\mfP$ we have the explicit geodesic 
$\gamma_s=(\Sigma,m,s^2 \kappa)$. 

Considering now the reduced Boltzmann entropy $\sfE(\gamma_s)$, we find
\[
e(s):=\sfE(\gamma_s)= \sfH(\Sigma,m|\Gamma,n) + \varkappa \lambda_\rmB \big(\frac{s^2 \kappa}\varkappa\big). 
\]
Since $ \ddot e(s) = 4\kappa + 2\kappa\log\big(s^2\kappa/\varkappa\big)\to - \infty$ for $s\searrow 0$, we see that geodesic $\uplambda$-convexity fails. 

We may still have a restricted version of sublevel $\uplambda$-convexity if we only consider sublevels that do not contain the trivial measure $\mu=0$. More precisely, if $\target=\varkappa G(\Gamma,n)$, then $\calH_\rmB(0|\target)=\varkappa$, and it can be shown that for $E\in (0,\varkappa)$, there exists $\wt\uplambda(E)$ such that for $\sfp=(\Sigma,m,\kappa)\in \mfP$ with $\sfE(\sfp)\leq E$, we have $\HESS_{\bbK^\mathrm{red}} \sfE(\sfp)\geq \wt\uplambda(E) \bbK^\mathrm{red} (\sfp)$. 
\end{remark}

\section{Long-time asymptotic decay estimates}
\label{se:LongTime}

In this section, we investigate the long-time behavior of solutions to the $\HK_{\alpha,\beta}$-Boltzmann gradient flow over the Gaussian measures. 

\subsection{Polyak-\L ojasiewicz estimate for Gaussian target measures}
\label{su:PolLojEstim}

As seen above, there is little hope of obtaining good long-time asymptotic decay estimates for the $\big(\mfP_1,\sfE_1, \bbK^\mathrm{red}_{\alpha,\beta}\big)$-gradient flow solutions based on geodesic $\uplambda$-convexity, since we rarely have $\uplambda \gneqq 0$. However, exponential decay can be derived already from the weaker Polyak-\L ojasiewicz estimate---also called the PL or \emph{gradient dominance} condition (see \cite{karimiLinearConvergenceGradient2020,carrillo2024fisher}). 
Its global form reads:
\begin{equation}
    \label{eq:PolLoj.gen}
    \big\langle \rmD \sfE_1(\sfp), \bbK^\mathrm{red}_{\alpha,\beta}(\sfp) 
    \rmD \sfE_1(\sfp) \big\rangle \geq c_{\rmP\rmL} \,\sfE_1(\sfp) \quad \text{for all } \sfp \in \mfP_1,
\end{equation}
where $c_{\rmP\rmL} >0$ is the Polyak-\L ojasiewicz constant. 
Note that in our case $\inf_{p\in \mfP_1} \sfE_1(\sfp) =0$. 

Often, it is also useful to use a \emph{sublevel version of P\L{}}, namely 
\begin{equation}
    \label{eq:PolLoj.sublevel}
\begin{aligned}
&\forall \, E>0\ \exists\, c_{\rmP\rmL}(E)>0\ \forall\, \sfp\in\mfP_1\text{ with } 
  \sfE_1(\sfp)\leq E: \\  
  &\hspace{8em} \big\langle \rmD \sfE_1(\sfp), \bbK^\mathrm{red}_{\alpha,\beta}(\sfp) 
    \rmD \sfE_1(\sfp) \big\rangle \geq c_{\rmP\rmL}(E) \,\sfE_1(\sfp). 
\end{aligned}
\end{equation}
Here, we may assume that $E \mapsto c_{\rmP\rmL}(E) $ is non-decreasing. The long-time asymptotics can then be deduced by a simple application of Gr\"onwall's lemma to the energy-dissipation relation, giving the decay estimate 
\[
	\sfE_1(\sfp(0))\leq E \quad \Longrightarrow \quad \sfE_1(\sfp(t)) \leq \ee^{-{c_{\rmP\rmL}(E)} t} \sfE_1(\sfp(0)) \quad\text{for all $t>0$.}
\]

To apply the theory in the present case, we use the decomposition into the covariance and the mean part: 
\begin{align*}
	\sfE_1(\Sigma,m)= \sfH_\mafo{cov}(\Sigma) + \sfH_\rmm(m) \quad \text{with} \quad\left\{
	\begin{aligned}
	\sfH_\mafo{cov}(\Sigma) &= \frac12 \trace(\Gamma^{-1}\Sigma{-}I) -\frac12\log\big(\frac{\det\Sigma}{\det\Gamma}\big), \\
\sfH_\rmm(m) &= \frac12 ( m{-}n ) \cdot \Gamma^{-1} (m{-}n). 
	\end{aligned}\right.
\end{align*}
Correspondingly, the dissipation $\calD(\sfp) := \langle \rmD\sfE_1(\sfp), \bbK^\mathrm{red}_{\alpha,\beta}(\sfp)  \rmD \sfE_1(\sfp) \rangle$ has the splitting
\begin{align}
  &\calD(\sfp)=\calD_\mafo{cov}(\Sigma) + \calD_\rmm(\sfp) \quad 
  \\
  &\qquad\text{with}\quad\left\{
    \begin{aligned}
      \label{eq:Dissipation.split}
\calD_\mafo{cov}(\Sigma) &= \alpha \trace\big(( \Gamma^{-1}\Sigma{-}I)^2 
  \Sigma^{-1}\big) + \frac\beta2 \trace\big( ( \Gamma^{-1}\Sigma{-}I)^2\big) 
\\
\calD_\rmm(\sfp) &= \alpha | \Gamma^{-1} ( m {-} n)|^2 + \beta \, 
( m {-} n) \cdot \Gamma^{-1}\Sigma \Gamma^{-1} ( m {-} n).
\end{aligned}\right.
\end{align}

\begin{remark}\label{rem:H_cov_spd}
	The functional $\sfH_\mathrm{cov}$ may be expressed fully in terms of the eigenvalues of the symmetric positive definite matrix $ B=\Gamma^{-1/2} \Sigma \Gamma^{-1/2}\in\bbR^{d\times d}_{\spd}$. Indeed, denoting by $(b_i)_{i=1,...,d}$ its eigenvalues, we obtain
	\begin{align}\label{eq:H_cov_spd}
		\sfH_\mafo{cov}(\Sigma)= \frac12 \sum_{i=1}^d  \upphi(b_i) \qquad \text{with}\quad \upphi(b)= b-1-\log b = b\lambda_\rmB(1/b).
	\end{align}
	Notice that $\upphi\colon[0,+\infty)\to [0,+\infty]$ is a convex and lower semicontinuous function with
	\[
		0\le \upphi(y) \le \upphi'(y)(y-1) = (y^{1/2}-y^{-1/2})^2=:\upzeta_\upphi(y)\qquad\text{for all $y\in(0,+\infty)$}.
	\]
\end{remark}

Our first result provides a Polyak-\L ojasiewicz (P\L) estimate for the covariance functions only (i.e., it works for centered Gaussians). It is based on the function
\[
	h(\delta,\beta;J) := \inf_{y\in J} H(\delta,\beta; y) \quad \text{with}\quad 
H(\delta,\beta;y) := \begin{cases}
	(\delta + \beta y)\frac{\upzeta_\upphi(y)}{\upphi(y)} & \text{for }y\neq 1, \\
	2(\delta + \beta) &\text{for $y=1$},
\end{cases}
\]
for any open interval $J=(J_-,J_+)\subset(0,+\infty)$. For better intuition, we plot the auxiliary function $H$ in Figure~\ref{fig:h_aux_function}.
\begin{figure}[H]
  \centering
  \includegraphics[width=0.7\textwidth]{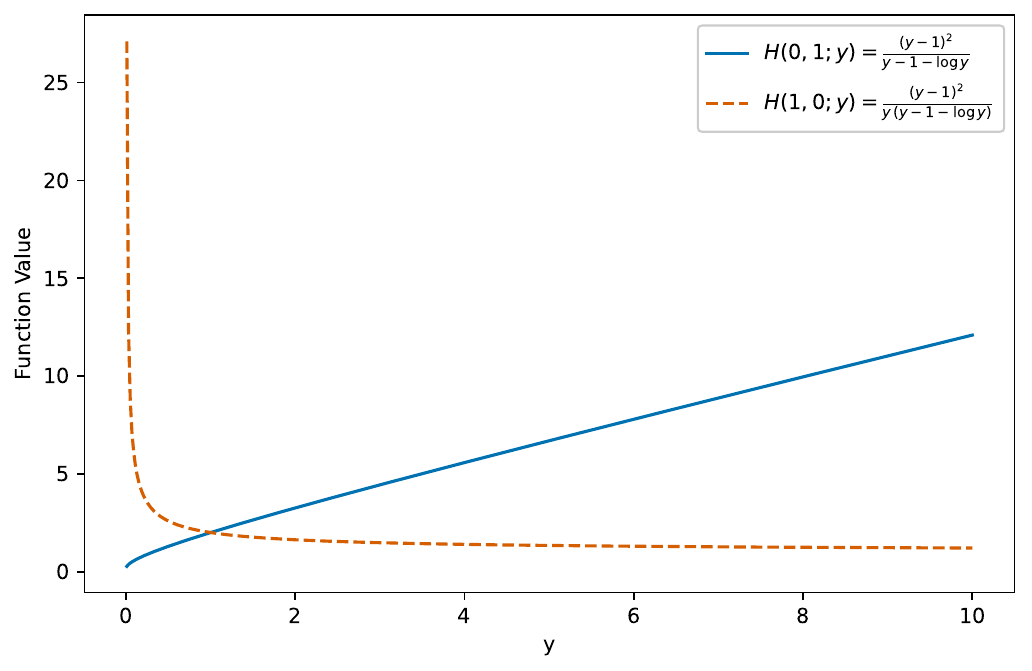}
  \caption{Plot of the auxiliary function $H(\delta,\beta;y)$ in Lemma~\ref{lem:PL-h}.}
  \label{fig:h_aux_function}
\end{figure}

\begin{lemma}
  \label{lem:PL-h}
	Let $J=(J_-,J_+)\subset(0,+\infty)$ with $ J_-<1< J_+ $, then
	\begin{gather*}
		h(1,0;J) = \lim_{y\to J_+}\frac{\upzeta_\upphi(y)}{\upphi(y)} 
         \in [1,2),\quad h(0,1;J) = \lim_{y\to J_-}\frac{y\upzeta_\upphi(y)}{\upphi(y)} \in [0,2 ) . 
	\end{gather*}
	In particular, we have $h(\delta,\beta;J) \ge \delta h(1,0;J) + 
    \beta h(0,1;J)$ for $\delta\beta>0$. 
	
	Moreover, if $J=(0,+\infty)$, then we have $h(1,0;J)=1$, $h(0,1;J)=0 $, 
    and \\
    $h(\delta,\beta; J) \geq  \min\{ \delta H(1,0;y_*), 
    \delta{+}\beta H(0,1;y_*)\}$ where $y_*>0$ is arbitrary.
\end{lemma}
\begin{proof}
	From Remark~\ref{rem:H_cov_spd}, we have that $\upzeta_\upphi(y) \ge \upphi(y)$ 
    for all $y\in(0,+\infty)$. Moreover, it is not difficult to see that 
    $y\mapsto \upzeta_\upphi(y)/\upphi(y)$ is a decreasing function, which gives 
    the first expression. As for the second, we similarly check that 
    $y\mapsto y\upzeta_\upphi(y)/\upphi(y)$ is an increasing function with 
	\[
		\lim_{y\to 0} \frac{y\upzeta_\upphi(y)}{\upphi(y)}= 0,\qquad \lim_{y\to +\infty} \frac{y\upzeta_\upphi(y)}{\upphi(y)}=+\infty.
	\]
	The second last estimate follows straightforwardly from the fact that $\inf (f+g) \ge \inf f + \inf g$. 
    
    For the final estimate, we use the monotonicity of $y \mapsto H(1,0;y)$
    and $y \mapsto H(0,1;y)$ and obtain 
    \begin{align*}
    H(\delta,\beta;y) &=\delta H(1,0;y)+\beta H(0,1;y) \geq \delta H(1,0;y_*) +  
    \beta \,0 \quad \text{for } y \leq y_*,
    \\
     H(\delta,\beta;y) &=\delta H(1,0;y)+\beta H(0,1;y) \geq \delta\,1 +  
    \beta H(0,1;y_*) \quad \text{for } y \geq y_*.
    \end{align*}
    This gives the final estimate.
\end{proof}

Based on Lemma~\ref{lem:PL-h}, our next proposition shows that a global P\L{}-inequality holds if and only if $\alpha_\Gamma :=\alpha\nu_\mafo{min}(\Gamma^{-1})>0$. However, the sublevel P\L{}-inequality holds even when $\alpha_\Gamma=0$.

\begin{proposition}[P\L{}-estimates for the covariance]
\label{prop:PLSigma} Set $ \alpha_\Gamma := \alpha\nu_\mafo{min}(\Gamma^{-1})$, then for all $ \Sigma \in \R^{d\ti d}_\spd$ we have the P\L{}-estimate 
\[
    \calD_\mafo{cov}(\Sigma) \geq \mfc_\mafo{cov}(\alpha, \beta) \, 
    \sfH_\mafo{cov}(\Sigma)\quad\text{with}\quad \mfc_\mafo{cov}(\alpha, \beta) = 
    \begin{cases}
        h(2\alpha_\Gamma,\beta; (0, \infty)) & \text{for $\alpha_\Gamma\beta>0$}, \\
         2 \alpha_\Gamma& \text{for $\beta=0$}
         ,
         \\
         0 & 
         \text{for $\alpha_\Gamma=0$}.
    \end{cases}
\]
In particular, $\mfc_\mafo{cov}(\alpha, \beta) \in [ 2\alpha_\Gamma, 
4\alpha_\Gamma{+}2 \beta)$.
Furthermore, $\mfc_\mafo{cov}(\alpha, \beta) >0$ if and only if $\alpha_\Gamma>0$.

On the sublevel $\sfH_\mafo{cov}(\Sigma)\leq E < \infty $, the
P\L{} estimate above holds with
\[
	\mfc_\mafo{cov}(\alpha, \beta;E) = 2\alpha_\Gamma h(1,0;J_E) + \beta h(0,1;J_E),
\]
where $J_E = ( \ee^{-(1+2E)} , 1+\log 4 + 4E)$.
\end{proposition}
\begin{proof}
We recall the auxiliary symmetric positive definite matrix $ B=\Gamma^{-1/2} \Sigma \Gamma^{-1/2}\in\bbR^{d\times d}_{\spd}$ and denote by $(b_i)_{i=1,...,d}$ its eigenvalues.
Then,
\begin{align*}   
 & \calD_\mafo{cov}(\Sigma) = \alpha \trace\big((B{-}I)^2 B^{-1} \Gamma\big) + \frac\beta2 \trace\big((B{-}I)^2\big) \\
 &\qquad \geq \frac12\sum_{i=1}^d \big( \delta/b_i
    + \beta \big)(b_i{-}1)^2 = \frac12\sum_{i=1}^d \big( \delta
    + \beta b_i \big)\upzeta_\upphi(b_i),\qquad \delta := 2\alpha_\Gamma.
\end{align*}
For the inequality above, we used $\alpha \trace\big((B{-}I)^2 B^{-1} \Gamma\big) \geq \alpha_\Gamma \sum_{i=1}^d (b_i{-}1)^2/b_i.$
The last equality follows from the definition $\upzeta_\upphi(b_i) = 
(b_i^{1/2} - b_i^{-1/2})^2
= 
(b_i{-}1)^2/b_i$ for all $b_i\in(0,+\infty)$.

Recalling the properties of $h, H$ provided in Lemma~\ref{lem:PL-h}, we obtain:
\begin{align}
  \calD_\mafo{cov}(\Sigma) 
    \geq  
    \frac12\left(\inf_{b_i\in J} H(\delta, \beta; b_i) \right)\sum_{i=1}^d \upphi(b_i)
    =
    h(\delta, \beta; J) \sfH_\mafo{cov}(\Sigma)
    .
\label{eq:PLSigma.1}
\end{align}
Using
$h( \delta,\beta;(0,\infty)) 
\geq \min\{ \delta H(1,0;1), \delta{+}\beta H(0,1;1)\} 
\geq \delta$
and
$ h(\delta,\beta;(0,\infty) )$ $\leq H(\delta,\beta;1)
=  2 (\delta{+}\beta)$,
we obtain the desired statement in the first part of the proposition
after checking the cases of
$\alpha_\Gamma=0, \beta=0,
\text{ and }\alpha_\Gamma\beta>0$.

As for the sublevel P\L{}-inequality, 
we exploit the analysis from the proof of Lemma~\ref{lem:Sublevel}, and replace the matrix $\Sigma, \Gamma$ by $B, I$.
Then, we use the argument of the last two estimates from the proof of Lemma~\ref{lem:Sublevel}, obtaining
\[
	\sfH_\mafo{cov}(\Sigma) \le E\quad\Longrightarrow\quad b_j \in  [\ee^{-(1+2E)},1+\log 4 + 4E]=:J_E,\quad j=1,\ldots,d.
\]
Applying Lemma~\ref{lem:PL-h} then yields the last assertion.
\end{proof}

\begin{remark}
If $\alpha \Gamma^{-1}\geq \beta I $, the constant $\mfc_\mafo{cov}(\alpha,\beta)$
is comparable to the smallest eigenvalue of $2\alpha \Gamma^{-1}+\beta I$, which
corresponds to the optimal decay rate as given by the explicit solution 
$\Sigma(t)=A(t)^{-1}$ (cf.\ \eqref{eq:SigmaConverges} in Remark 
\ref{re:Explicit.A(t).decay}).
For very small $\alpha_\Gamma$, the 
P\L{}-coefficient does not reflect the decay of the long-time asymptotics anymore, because the Hellinger geometry dominates. Starting with a highly localized Gaussian (i.e.\ $\Sigma$ has very small eigenvalues) the Hellinger gradient flow 
allows the solution to stay localized and hence far away from some target Gaussian with covariance $\Gamma>0$. Hence, diffusion generated from the Otto-Wasserstein geometry is necessary for widening the Gaussian.

Indeed, the lower bound $ \mfc_\mafo{cov}(\alpha, \beta) \geq 
 2\alpha_\Gamma$ in Proposition \ref{prop:PLSigma} can be significantly 
 improved for $\alpha_\Gamma \ll \beta$. To see this, we make the choice $y_*= 2\alpha_\Gamma/\beta\ll 1$ in 
 the 
 last assertion of Lemma~\ref{lem:PL-h}.
 We find
 $h(\delta,\beta; J) \geq  \min\{ \delta H(1,0;y_*), 
 \delta{+}\beta H(0,1;y_*)\}
 \approx \beta / \log (2\alpha_\Gamma/\beta)$.
 Since we have $ y * \gg \log (1/y_*) $ for $y_* \ll 1$, we find the desired estimate
 $$
 \mfc_\mafo{cov}(\alpha, \beta)
 =h(\delta ,\beta; (0, \infty))
 \approx \beta / \log (2\alpha_\Gamma/\beta)
 \gg 2\alpha_\Gamma
 .
 $$
\end{remark}

The treatment of the functionals controlling the mean is simpler due to their quadratic nature. Note that we cannot take advantage of the Hellinger part for a global P\L{}-estimate because of the factor $\Sigma$ in the dissipation~\eqref{eq:Dissipation.split} term $\calD_\rmm(\sfp)$, which does not have a lower bound. However, as in the previous case, the P\L{}-coefficient improves when restricted to suitable sublevels of $\sfH_\mafo{cov}$.

\begin{proposition}[P\L{}-estimates for the mean]
\label{prop:PLMean}
The global P\L{}-estimate
\[
	\calD_\rmm(\Sigma,m) \geq \mfc_\rmm(\alpha,\beta) \sfH_\rmm(m)\quad\text{with}\quad \mfc_\mafo{m}(\alpha,\beta) = 2\alpha_\Gamma.
\]
If, in addition, $\sfH_\mafo{cov}(\Sigma)\le E$, then the P\L{}-estimate above holds with
\[
	\mfc_\mafo{m}(\alpha,\beta;E) = 2\alpha_\Gamma + 2\beta \ee^{-(1+2E)}.
\]
\end{proposition}
\begin{proof}
The proof is similar to that of Proposition~\ref{prop:PLSigma}:
The lower bound on the Hellinger part holds since the eigenvalues $(b_i)_{i=1,\ldots,d}$ of $B=\Gamma^{-1/2} \Sigma \Gamma^{-1/2}$ satisfy $b_i \geq \ee^{-(1+2E)}$. Hence, the second term in the expression of $\calD_\rmm(\sfp)$
in \eqref{eq:Dissipation.split} satisfies
\begin{align*}
  \beta (m{-}n) \cdot \Gamma^{-1}\Sigma \Gamma^{-1} ( m {-} n) \geq \beta \ee^{-(1+2E)} (m{-}n) \cdot \Gamma^{-1} ( m {-} n)
  = 2 \beta \ee^{-(1+2E)} \sfH_\mafo{cov}(\Sigma).
\end{align*}
The remaining part of the proof is straightforward and similar to that of Proposition~\ref{prop:PLSigma}, and is thus omitted.
\end{proof}

In summary, we obtain the following decay estimate based on the P\L{} estimates.

\begin{theorem}[Sublevel P\L{} for $\HK$-Boltzmann flow with Gaussian target]
\label{thm:PL.Decay}
Let\\
$\target=\sfG(\Gamma,n)$ be a Gaussian target measure. Then, the solution $t \mapsto \sfp(t)=(\Sigma(t),m(t))$ to the $\HK$-Boltzmann Gaussian gradient system $(\mfP_1,\sfE_1, \bbK^\mathrm{red}_{\alpha,\beta})$ with $\sfE_1(\sfp(0))\le E$ enjoys the decay estimate
\begin{align*} 
	\sfE_1(\sfp(t)) \leq \ee^{-\mfc_\mafo{cov}(\alpha,\beta;E)t} \sfH_\mafo{cov}(\Sigma(0)) + \ee^{-\mfc_\mafo{m}(\alpha,\beta;E) t} \sfH_\rmm(m(0)) \qquad\text{for all $t\ge0$},
\end{align*}
with the decay rates $\mfc_\mafo{cov}$ and $\mfc_\mafo{m}$ given in Propositions~\ref{prop:PLSigma} and \ref{prop:PLMean}, respectively.
\end{theorem}

Based on the asymptotics of the explicit solution given in Remark~\ref{re:Explicit.A(t).decay}, one expects that the decay rate in the case $\alpha_\Gamma=0$ to be $\beta>0$ and independent of the initial data. Clearly, the rates we obtained above is slower than $\beta$, especially when $E\gg  0$. In the next section, we see how the expected decay rate can be obtained by a further analysis on the evolution of the eigenvalues of $\Gamma^{-1/2}\Sigma\Gamma^{-1/2}$.

\subsection{Refined decay estimate for Gaussian target measures}

As in the previous section, we begin our exposition by studying the long-time asymptotics for the covariance $\Sigma$. We recall that
\[
	\frac{\rmd }{\rmd t} \sfH_\mafo{cov}(\Sigma) = -\calD_\mafo{cov}(\Sigma) = -\frac12\sum_{i=1}^d \big( 2\alpha_\Gamma
    + \beta b_i \big)\upzeta_\upphi(b_i),
\]
where $(b_i)_{i=1,\ldots,d}$ are the eigenvalues of the matrix $B=\Gamma^{-1/2}\Sigma \Gamma^{-1/2}\in\R^{d\ti d}_{\spd}$. Therefore, one way of improving the decay rate is to understand the evolution of the eigenvalues $b_i$. 

From the evolution for $\Sigma$ (cf.\ \eqref{eq:ODE66.mfP}), we deduce the evolution for $B$:
\[
	\frac{\rmd}{\rmd t}B = -\alpha \bigl(\Gamma^{-1}B + B\Gamma^{-1} - 2\Gamma^{-1}\bigr) + \beta\bigl(B - B^2\bigr),\qquad B(0)=\Gamma^{-1/2}\Sigma(0)\Gamma^{-1/2}.
\]
Using the Hellmann-Feynman theorem \cite{hellmann2015hans}, we then obtain an equation for the eigenpairs $(b_i,v_i)_{i=1,\ldots,d}$ of $B$, given by
\begin{align}\label{eq:gaussian_eigen}
	\dot b_i = -\bigl( 2\alpha\langle v_i | \Gamma^{-1}v_i\rangle + \beta b_i\bigr)(b_i-1),\qquad i=1,\ldots,d.
\end{align}

\begin{lemma}\label{lem:gaussian_eigen_bound}
	The evolution of eigenvalues $(b_i)_{i=1,\ldots,d}$ for $B$ satisfy the estimates
	\begin{align}\label{eq:gaussian_eigen_bound}
		1\wedge \mfb_i^{-1}(t) \le b_i(t) \le 1\vee \mfb_i(t),
    \qquad \mfb_i(t):= 1 + (b_i(0)-1)\,\ee^{-\upnu_\mafo{cov}(\alpha,\beta) t},
	\end{align}
	with $\upnu_\mafo{cov}(\alpha,\beta):=2\alpha \nu_\mafo{min}(\Gamma^{-1}) + \beta$.
\end{lemma}
\begin{proof}
	Notice that $b=1$ is a stationary state for \eqref{eq:gaussian_eigen}. Splitting the initial values into two sets $A^+:= \{i:b_i(0)\ge 1\}$ and $A^-:=\{i:0< b_i(0)< 1\}$, we easily deduce that the evolution leaves these sets invariant. Hence, for $i\in A^+$, we obtain
\begin{align*}
	\frac{\rmd}{\rmd t}(b_i-1) &= -\bigl( 2\alpha\langle v_i|\Gamma^{-1}v_i\rangle + \beta b_i\bigr)(b_i-1) \le -\upnu_\mafo{cov}(\alpha,\beta)(b_i-1),
\end{align*}
and consequently
\[
	b_i(t) \le 1 + (b_i(0)-1)\,\ee^{-\upnu_\mafo{cov}(\alpha,\beta) t} = \mfb_i(t)\qquad \text{for all $i\in A^+$.}
\]
Similarly, one obtains for $i\in A^-$,
\[
	\frac{\rmd}{\rmd t}(b_i^{-1}-1) = -\bigl( 2\alpha\langle v_i|\Gamma^{-1}v_i\rangle b_i^{-1} + \beta \bigr)(b_i^{-1}-1) \le -\upnu_\mafo{cov}(\alpha,\beta)(b_i^{-1}-1),
\]
and therefore,
\[
	b_i(t) \ge \frac{1}{1+ (b_i^{-1}(0)-1)\,\ee^{-\upnu_\mafo{cov}(\alpha,\beta) t}} = \mfb_i^{-1}(t)\qquad \text{for all $i\in A^-$,}
\]
thus concluding the proof.
\end{proof}

Recalling from Remark~\ref{rem:H_cov_spd} that $\upphi(b)\le \upzeta_\upphi(b)$ for all $b>0$, we have that 
\[
	\frac{\rmd }{\rmd t} \sfH_\mafo{cov}(\Sigma) = -\calD_\mafo{cov}(\Sigma) = -\big( 2\alpha_\Gamma
    + \beta b_\mafo{min} \big)\sfH_\mafo{cov}(\Sigma),
\]
and from which we obtain
\[
	\sfH_\mafo{cov}(\Sigma(t)) \le \ee^{-\upnu_\mafo{cov}(\alpha,\beta)t}\,\sfH_\mafo{cov}(\Sigma(0)) \cdot\begin{cases}
		1 & \text{if $b_\mafo{min}(0)\ge 1$,}\\
		\displaystyle\exp\left(
    \beta \int_0^t (1-b_\mafo{min}(s))\dd s\right) & \text{if $b_\mafo{min}(0)< 1$,}
	\end{cases}
\]
where $b_\mafo{min}$ is the minimal eigenvalue of $B$. The integral term on the right-hand side may be further estimated using Lemma~\ref{lem:gaussian_eigen_bound},
\begin{align*}
  1-b_\mafo{min}(s) \le 1 - \mfb_i^{-1}(s)
  =
  \frac{(b_i^{-1}(0)-1)\,\ee^{-\upnu_\mafo{cov}(\alpha,\beta) t}}{1+ (b_i^{-1}(0)-1)\,\ee^{-\upnu_\mafo{cov}(\alpha,\beta) t}}
\end{align*}
Let:
$$
G(t) : =  -\frac{1}{\upnu_\mafo{cov}(\alpha,\beta)} \log \Bigl( 1 + (b_\mafo{min}(0)-1)\ee^{-\upnu_\mafo{cov}(\alpha,\beta) t}\Bigr),
$$
we can verify that 
$$
\int_0^t (1-b_\mafo{min}(s))\dd s \le
\int_0^t (1 - \mfb_i^{-1}(s)) \dd s
=
G(t) - G(0)
.
$$
Thus, we obtain
\[
	\int_0^t (1-b_\mafo{min}(s))\dd s \le
   -\frac{1}{\upnu_\mafo{cov}(\alpha,\beta)} \log \Bigl( b_\mafo{min}(0) + (1-b_\mafo{min}(0))\ee^{-\upnu_\mafo{cov}(\alpha,\beta) t}\Bigr)
   .
\]
Since $(1-b_\mafo{min}(0))\ee^{-\upnu_\mafo{cov}(\alpha,\beta) t}>0$, 
the above estimate yields
\[
	\exp\left(\beta \int_0^t (1-b_\mafo{min}(s))\dd s\right) 
	\le \bigl( b_\mafo{min}(0)\bigr)^{- \frac{\beta}{\upnu_\mafo{cov}(\alpha,\beta)}}.
\]
From these estimates, we conclude that the solution $t \mapsto \Sigma(t)$ to the covariance equation \eqref{eq:ODE66.mfP} enjoys the decay estimate
\begin{align}\label{eq:decay_cov}
	\sfH_\mafo{cov}(\Sigma(t)) \le \bigl(\min\{1,b_\mafo{min}(0)\}\bigr)^{-\frac{\beta}{\upnu_\mafo{cov}(\alpha,\beta)}} \ee^{- \upnu_\mafo{cov}(\alpha,\beta)t}\,\sfH_\mafo{cov}(\Sigma(0)),
\end{align}
with the decay rate $\upnu_\mafo{cov}(\alpha,\beta)=2\alpha \nu_\mafo{min}(\Gamma^{-1}) + \beta$.

\medskip
The decay of the mean can be easily established from the differential inequality
\begin{align*}
	\frac{\rmd}{\rmd t}\sfH_\mafo{m}(m) = - \calD_\mafo{m}(\Sigma,m) 
\le -2\bigl(\alpha\nu_\mafo{min}(\Gamma^{-1}) + \beta b_\mafo{min}\bigr)\sfH_\mafo{m}(m).
\end{align*}
Repeating the approach used for the decay of the covariance, we find the decay estimate
\begin{align}\label{eq:decay_m}
	\sfH_\mafo{m}(m(t)) \le \bigl(\min\{1,b_\mafo{min}(0)\}\bigr)^{-\frac{2\beta}{\upnu_\mafo{m}(\alpha,\beta)}} \ee^{- \upnu_\mafo{m}(\alpha,\beta)t}\,\sfH_\mafo{m}(m(0)),
\end{align}
with the decay rate $\upnu_\mafo{m}(\alpha,\beta)=2(\alpha \nu_\mafo{min}(\Gamma^{-1}) + \beta)$.

\medskip
We summarize our discussion in this section with the following statement.

\begin{theorem}[Refined estimate: $\HK$-Boltzmann with Gaussian target]
\label{thm:Decay}
Let $\target=\sfG(\Gamma,n)$ be a Gaussian target measure. Then the solution $t \mapsto \sfp(t)=(\Sigma(t),m(t))$ to the  Gaussian gradient system $(\mfP_1,\sfE_1, \bbK^\mathrm{red}_{\alpha,\beta})$ enjoys the decay estimate
\begin{align*} 
	\sfE_1(\sfp(t)) \leq \ee^{-\upnu_\mafo{cov}(\alpha,\beta)t} \sfH_\mafo{cov}(\Sigma(0)) + \ee^{-\upnu_\mafo{m}(\alpha,\beta) t} \sfH_\rmm(m(0)) \qquad\text{for all $t\ge0$},
\end{align*}
with the decay rates $\upnu_\mafo{cov}$ and $\upnu_\mafo{m}$ given as in \eqref{eq:decay_cov} and \eqref{eq:decay_m}, respectively.
\end{theorem}
The refined estimate from Theorem~\ref{thm:Decay} is sharper than the one from Theorem~\ref{thm:PL.Decay}
as it no longer depends on the initial condition $E$ for the sublevels.

\subsection{Generalization to strongly log-\texorpdfstring{$\uplambda$}{lambda}-concave target measures}

We now consider a sufficiently regular log-$\uplambda$-concave target measure $\target=\ee^{-V}\rmd x\in\calP(\R^d)$, i.e.,
there exists $\uplambda >0$ such that $V:\R^d\to \R$ is $\uplambda$-convex:
	\begin{align}\label{eq:V_lambda_convex}
		x\mapsto V(x) - \frac{\uplambda}{2}|x|^2 \quad\text{is convex}.
	\end{align}
We recall that for $\sfp=(\Sigma,m)\in \mfP_1$,
\begin{align*}
	\sfE_1(\sfp) = \calH_\rmB(\sfG(\sfp)|\target)
	= -\frac{1}{2}\log \det \Sigma - \frac{d}{2}\log 2\target - \frac{d}{2} + \int_{\R^d} V\sfG(\sfp)\dd x.
\end{align*}

\begin{remark}\label{rem:W2_lambda_convexity}
Since $\target$ is log-$\uplambda$-concave, the driving energy $\calE=\calH_\rmB(\cdot|\target)$ is geodesically $\uplambda$-convex w.r.t.\ the $\bbW_2$-distance \cite{AmGiSa05GFMS}, i.e., for any $\bbW_2$-geodesic $[0,1]\ni \tau\mapsto \rho_\tau\in\calP(\R^d)$,
\[
	\calE(\rho_\tau) \le (1{-}\tau)\calE(\rho_0) + \tau \calE(\rho_1) - \frac{\uplambda}{2}(1{-}\tau)\tau  \bbW_2^2(\rho_0,\rho_1) ,\quad \tau\in[0,1].
\]
In particular, it is $\uplambda$-convex w.r.t.\ the $\bbW_2$-distance restricted to Gaussian measures (also known as the Bures-Wasserstein distance \cite{lambertVariationalInferenceWasserstein2022}), which is equivalent to the distance induced by $\bbK_\Ot^\text{red}$ (cf.\ Remark ). In this case, we find
\begin{align}\label{eq:energy-slope}
	\sfE_1(\sfp) - \inf_{\sfq\in\mfP_1} \sfE_1(\sfq) \le \frac{1}{2\uplambda}|\partial_\Ot\sfE_1|^2(\sfp),
\end{align}
where $|\partial_\Ot\sfE_1|$ is the metric slope of $\sfE_1$ w.r.t.\ the Bures-Wasserstein distance.
\end{remark}

The existence of unique minimizers for $\sfE_1$ on $\mfP_1$ is known. Nevertheless, we provide a direct proof of this while also providing insights into the sublevels of $\sfE_1$.

\begin{proposition}\label{prop:sublevel-minimizer}
	Assume condition \eqref{eq:V_lambda_convex}. Then for any $E>0$, there exists an $r_E>0$ such that for every $\sfp=(\Sigma,m)$ with $\sfE_1(\sfp)\le E$,
	\[
		r_E^{-1}\le \sigma_i \le r_E\quad \text{for all $i=1,\ldots,d$},\qquad |m|\le r_E,
	\]
	with $(\sigma_i)_{i=1,\ldots,d}$ being the eigenvalues of $\Sigma$. 
	
	In particular, the functional $\sfE_1\colon \mfP_1\to [0,+\infty] $ admits a unique minimizer, which satisfies the Euler-Lagrange equations
	\[
		\Sigma = \int \nabla_x^2 V\,\sfG(\Sigma,m)\dd x,\qquad 0=\int \nabla_xV \,\sfG(\Sigma,m)\dd x.
	\]
\end{proposition}
\begin{proof}
	Clearly $\sfE_1$ is continuous on $\mfP_1$. We now prove the coercivity of $\sfE_1$. 
	
	Consider the sublevel set $L_E :=\{\sfp\in\mfP_1 : \sfE_1(\sfp)\le E\}$. Since $V$ is $\uplambda$-convex,
	\[
		V(x) - \frac{\uplambda}{2}|x|^2 - V(\ol x) \ge \nabla V(\ol x)\cdot (x-\ol x) \ge -|\nabla V(\ol x)||\ol x|
		-\frac{|\nabla V(\ol x)|^2}{\uplambda} - \frac{\uplambda}{4}|x|^2,
	\]
	there exists a constant $c_E>0$ such that
	\begin{align*}
	\sfE_1(\sfp) &\ge -\frac{1}{2}\log \det \Sigma + \frac{\uplambda}{4}\int_{\R^d} |x|^2\,\sfG(\sfp)\dd x - c_E \\
	&= -\frac{1}{2}\log \det \Sigma + \frac{\uplambda}{4}\bigl(\trace[\Sigma] + |m|^2\bigr) - c_E \\
	&= \frac{1}{2}\sum_{i=1}^d \Bigl(\log\sigma_i^{-1} + \frac{\uplambda}{2}\sigma_i\Bigr) + \frac{\uplambda}{4}|m|^2 - c_E =: \frac{1}{2}\sum_{i=1}^d \uptheta(\sigma_i) + \frac{\uplambda}{4}|m|^2 - c_E,
	\end{align*}
	with $(\sigma_i)_{i=1,\ldots,d}$ being the eigenvalues of $\Sigma$. It is easy to check that $\uptheta(r)=\log r^{-1} + \uplambda r/2$ is convex and has the properties (1) $\uptheta(r)\ge 1 - \log 2 +\log\uplambda$, (2) $\lim_{r\to 0} \uptheta(r) = +\infty$ and $\lim_{r\to +\infty} \uptheta(r) = +\infty$. In particular, this implies the existence of some constant $r_E>0$ such that for every $\sfp=(\Sigma,m)\in L_E$:
	\[
		r_E^{-1}\le \sigma_i \le r_E\quad \text{for all $i=1,\ldots,d$},\qquad |m|\le r_E,
	\]
	from which we deduce the coercivity of $\sfE_1$. The direct method of Calculus of Variations then gives the existence of a minimizer of $\sfE_1$.
	
	The uniqueness of minimizers easily follows from the $\bbW_2$-geodesic $\uplambda$-convexity of $\calE$. Indeed, let $\sfp_0,\sfp_1\in\mfP_1$ be minimizers of $\sfE_1$ and consider a $\bbW_2$-geodesic $[0,1]\ni \tau \mapsto \rho_\tau$ with $\rho_0=\sfG(\sfp_0)$, $\rho_1=\sfG(\sfp_1)$. Since $\bbW_2$-geodesics preserve Gaussianity, we have that $\rho_\tau = \sfG(\sfp_\tau)$ for some $\sfp_\tau\in\mfP_1$ for all $\tau\in (0,1)$. The $\bbW_2$-geodesic $\uplambda$-convexity of $\calE$ yields
	\[
		\sfE(\sfp_{\frac{1}{2}}) = \calE(\sfG(\sfp_\frac{1}{2})) \le \min_{q\in\mfP_1} \sfE(\sfq)  - \frac{\uplambda}{8}\bbW_2^2(\sfG(\sfp_0),\sfG(\sfp_1)) < \min_{q\in\mfP_1} \sfE(\sfq),
	\]
	which leads to a contradiction unless $\sfG(\sfp_0)=\sfG(\sfp_1)$, or equivalently $\sfp_0=\sfp_1$.
	
	The Euler-Lagrange equations can be easily computed and are left to the reader.
\end{proof}

One way of obtaining exponential decay of the $\SHe$-flow is to use the energy-dissipation inequality \eqref{eq:energy-slope}---this approach appeared in \cite{chenGradientFlowsSampling2023,lambertVariationalInferenceWasserstein2022}---which we now outline. Taking the temporal derivative of $\sfE_1$ along \eqref{eq:general-target} gives
\[
	\frac{\dd}{\dd t}\sfE_1(\sfp) = - \alpha |\partial_\Ot \sfE_1|^2(\sfp) - \beta |\partial_\SHe \sfE_1|^2(\sfp),
\]
where the slopes are given by 
\begin{align*}
	|\partial_\Ot \sfE_1|^2(\sfp) &= \trace\bigl[(I-\Gamma^{-1}\Sigma)^2\Sigma^{-1}\bigr] + \left\langle \int\nabla V \sfG(\sfp)\dd x,\int\nabla V \sfG(\sfp)\dd x\right\rangle, \\
	|\partial_\SHe \sfE_1|^2(\sfp) &= \trace\bigl[(I-\Gamma^{-1}\Sigma)^2\bigr] + \left\langle \int\nabla V \sfG(\sfp)\dd x,\Sigma\int\nabla V \sfG(\sfp)\dd x\right\rangle.
\end{align*}
In particular, we see that $|\partial_\Ot \sfE_1|^2(\sfp) \le \sigma_\mafo{min}|\partial_\SHe \sfE_1|^2(\sfp)$, where $\sigma_\mafo{min}$ is the minimal eigenvalue of $\Sigma$.
Applying the energy-dissipation inequality \eqref{eq:energy-slope},
we obtain
\begin{multline*}
  \frac{\dd}{\dd t}\sfE_1(\sfp)
  =
- \alpha |\partial_\Ot \sfE_1|^2(\sfp) - \beta |\partial_\SHe \sfE_1|^2(\sfp)
\le
- \left(\alpha + \beta \sigma_\mafo{min}^{-1}\right)
|\partial_\Ot \sfE_1|^2(\sfp)
\\
{\leq}
- 2 \lambda
\left(\alpha + \beta \sigma_\mafo{min}^{-1}\right)
\left(
  \sfE_1(\sfp) - \inf_{\sfq\in \mfP_1} \sfE_1(\sfq)
\right)
 \le - \uplambda\bigl(2\alpha + \beta r_E^{-1}\bigr) \Bigl(\sfE_1(\sfp) - \inf_{\sfq\in \mfP_1} \sfE_1(\sfq)\Bigr),
\end{multline*}
where $r_E>0$ is the constant provided by Proposition~\ref{prop:sublevel-minimizer} with $E=\sfE_1(\sfp(0))$. 
\EEE
\medskip

A direct application of Gronwall's inequality provides the overall decay estimate, as summarized in the following theorem.

\begin{theorem}[Decay estimate: $\HK$-Boltzmann with log-$\uplambda$-concave target]\label{thm:decay-log-lambda}
	Let\\
  $\target$ be a log-$\uplambda$-concave target measure satisfying \eqref{eq:V_lambda_convex}. Then the solution $t \mapsto \sfp(t)=(\Sigma(t),m(t))$ to the Gaussian gradient system $(\mfP_1,\sfE_1, \bbK^\mathrm{red}_{\alpha,\beta})$ enjoys the decay estimate
	\[
		\sfE_1(\sfp(t)) \le \ee^{-\upgamma(\alpha,\beta;E) t}\sfE_1(\sfp(0)) + (1-\ee^{-\upgamma(\alpha,\beta;E) t})\inf_{\sfq\in\mfP_1} \sfE_1(\sfq),
	\]
	with the decay rate $\upgamma(\alpha,\beta;E)=\uplambda(2\alpha + \beta r_E^{-1})$.
\end{theorem}

Notice that the decay rate corresponding to the Hellinger term depends on $\uplambda$, which contrasts with the case of Gaussian target measures. 
We leave more refined analysis than Theorem~\ref{thm:decay-log-lambda} for future work.

\section{Numerical algorithms and experiments}
\label{sec:NumericalAlgorithms}
\subsection{A discrete-time \texorpdfstring{$\HK$}{HK}-Gaussian gradient descent algorithm}
We now consider a discrete-time scheme for simulating the $\HK$-Gaussian gradient flow over Gaussian measures.
In essence,
our approach corresponds to
a gradient descent scheme
discretizing the
flow equation \eqref{eq:HK.Gauss.mass-shape-ode-all}
by interleaving the following Fisher-Rao \cite{amariMethodsInformationGeometry2000}, Bures-Wasserstein \cite{lambertVariationalInferenceWasserstein2022}, and mass update steps.

\paragraph{Fisher-Rao step}
We first derive the discrete-time update algorithm by discretizing the Fisher-Rao-Gaussian system for a target measure $\target\propto \ee^{-V}\rmd x$, which we recall here: 
\begin{subequations}
  \label{eq:Fisher-Rao-Gaussian}
\begin{align}
  \dot \Sigma & =
      \beta \left(\Sigma 
      - \Sigma \Gamma^{-1} \Sigma\right),
  \label{eq:Fisher-Rao-covariance-update}\\
  \dot m  & = -\beta\int_{\R^d} \Sigma\, \nabla V \sfG(\sfp)\dd x
  ,
  \label{eq:Fisher-Rao-mean-update}
\end{align}
where we use the notation $\Gamma^{-1} := \int_{\R^d} \nabla^2 V\sfG(\sfp)\dd x$. It can be interpreted as the inverse of the target covariance matrix.
\end{subequations}
Using an explicit Euler discretization, we derive the following update scheme:
\begin{align*}
	\Sigma_{k+1}^{-1} &\gets (1-\tau\beta) \Sigma_k^{-1} + \tau\beta \Gamma_k^{-1},\\
	m_{k+1} &\gets m_k + \tau \beta\Sigma_k \int_{\R^d} \nabla V\sfG(\sfp_k)\dd x,\qquad \sfp_k=(\Sigma_k,m_k),
\end{align*}
where we use the notation $\Gamma^{-1}_k := \int_{\R^d} \nabla^2 V\sfG(\sfp_k)\dd x$;
$\tau>0$ is the step size of the discrete-time scheme.
Note that, in numerical implementation, we only need to keep track of the precision matrix $\Sigma^{-1}$ instead of the covariance.

\paragraph{Bures-Wasserstein step}
The Bures-Wasserstein ODE takes the form
\begin{subequations}
  \label{eq:Bures-Wasserstein-ODE}
\begin{align}
    \dot  \Sigma &= 
    \alpha\left(2 I - \Sigma\Gamma^{-1} + \Gamma^{-1}\Sigma\right),
    \\
    \dot m  & = -\alpha\int_{\R^d} \nabla V \sfG(\sfp)\dd x.
  \end{align}
\end{subequations}
Then, the discrete-time update scheme is given by
\begin{align*}
	\Sigma_{k+1} &\gets \bigl(I + \tau\alpha \bigl(\Sigma_k^{-1} - \Gamma_k^{-1} \bigr)\bigr) \Sigma_k\bigl(I + \tau\alpha \bigl(\Sigma_k^{-1} - \Gamma_k^{-1} \bigr)\bigr),\\
	m_{k+1} &\gets m_k - \tau \alpha\int_{\R^d} \nabla V \sfG(\sfp_k)\dd x, 
\end{align*}

\paragraph{Mass update}
The discrete-time mass update step can also be straightforwardly derived from the mass ODE \eqref{eq:HK.mass-ode}, yielding the scheme
\begin{align*}
  \kappa_{k+1} \gets \kappa_k - \tau \beta  \kappa_k 
      \left(\calH(\rho_k|\target) + \log \kappa_k\right)
  .
\end{align*}

\paragraph{Estimating the integrals}
When the target measure $\target$ is a general non-Gaussian measure, the estimation of the integrals $\int_{\R^d} \nabla V \sfG(\sfp)\dd x$ and $\int_{\R^d} \nabla ^2 V \sfG(\sfp)\dd x$
can be done via Monte Carlo sampling by generating $n$ samples $X_k^i \sim \sfG(\sfp_k)$ for $i=1,\ldots,n$.
Then, the Monte Carlo estimate of the integrals is given by
\begin{align*}
  \int_{\R^d}  \nabla ^s V  \sfG(\sfp_k)\dd x
  \approx 
  \frac1n \sum_{i=1}^n \nabla ^s V(X_k^i) \text{ for } s = 0,1,2
  .
\end{align*}
We outline the concrete update scheme in Algorithm~\ref{alg:JKO}. For simplicity, we omit the superscript $i$ for the $i$-th sample in the algorithm box.
As this paper is not concerned with numerics and computational efficiency, we use small constant step sizes for $\tau$ throughout the experiments.
\begin{algorithm}
  \caption{Discrete-time HK-Gaussian Gradient Descent}
  \begin{algorithmic}[1]
  \State \textbf{Data:}
  initial mean $m_0$ and covariance $\Sigma_0$,
  step size $\tau  > 0$,
  potential function $V$ from the target $\target \propto \exp(-V)$ as an oracle
  \For{$k = 1,\ldots,N$}
      \State draw sample $X_k \sim \sfG(\sfp_k)$
      \State $\Sigma_{k+\frac12} \gets (I + \tau\alpha( \Sigma_k^{-1} + \nabla^2 V(X_k) ))\Sigma_k (I + \tau\alpha( \Sigma_k^{-1} + \nabla^2 V(X_k) ))$
      \Comment{BW: cov.}
      \State
      $m_{k+\frac12} \gets m_k - \tau \alpha
         \nabla V(X_k)
           $
           \Comment{BW: mean}
      \State (optional) re-sample $X_k \sim  \sfG(\sfp_{k+\frac12})=\sfG(\Sigma_{k+\frac12},m_{k+\frac12})$
      \State
           $\Sigma_{k+1}^{-1}
           \gets (1-\tau\beta) \Sigma_{k+\frac12}^{-1} + \tau\beta \nabla^2 V(X_k)$
           \Comment{FR: cov.}
      \State
      $m_{k+1} \gets m_{k+\frac12} - \tau \beta
        \Sigma_{k+1} \nabla V(X_k)
           $
           \Comment{FR: mean}
      \State 
      $\kappa_{k+1} \gets \kappa_k - \tau \beta  \kappa_k 
      \left(\calH(\rho_k|\target) + \log \kappa_k\right)
      $
      \Comment{Hellinger: mass}
  \EndFor
  \end{algorithmic}
  \label{alg:JKO}
\end{algorithm}

\subsection{Numerical experiments}
For notational convenience, we will use the variable $\theta\in \R^d$ to denote the base space, which is typically the parameter we wish to estimate in computational problems.
We can view the general risk minimization problem in optimization and machine learning
\begin{align}
  \min_\theta \mathbb E_{\xi \sim P_\text{data}}\ell (\theta, \xi)
  \label{eq:risk_minimization}
\end{align}
through the lens of the gradient flows in this paper.
The random variable $\xi$ represents the training data for the learning problem
and the probability distribution $P_\text{data}$ represents the data generating process.

A probabilistic perspective of the learning problem is given by lifting the minimization variable $\theta$ to the probabilistic distribution.
This is by considering the risk function $\mathbb E_{\xi \sim P_\text{data}}\ell (\theta, \xi)$ as the potential function over the decision variable $\theta$ (also referred to as the learning model).
Then, defining the linear energy functional
$V(\mu) = \int \mathbb E_{\xi \sim P_\text{data}}\ell (\theta, \xi) \mu(d\theta)$
for $\mu \in \calP(\R^d)$,
we can consider the entropy-regularized functional 
\begin{align*}
  F_\lambda(\mu) = \frac1\lambda V(\mu) + 
  \int  \log \mu(\theta) \dd \mu (\dd \theta ) 
  = \calH_\rmB\left(\mu(\theta)\Bigg| 
  \frac1{z}
          \exp{\left(-\frac{\mathbb E_{ P_\text{data}}\ell (\theta, \xi)}{\lambda}\right)}            
        \right)
  ,
\end{align*}
where $\lambda$ is the regularization parameter and $z$ is the normalization constant.
In general, we assume $\mu$ to be absolutely continuous w.r.t. the Lebesgue measure.
Then, the target measure satisfies
\begin{align*}
  \target(\theta) \propto 
    \exp{\left(-\frac{\mathbb E_{ P_\text{data}}\ell (\theta, \xi)}{\lambda}\right)}
  \quad \text{and} \quad
  \mathbb E_{ P_\text{data}}\ell (\theta, \xi) = - \lambda \log \target(\theta)+\text{const}
\end{align*}
The machine learning risk minimization problem can be cast as the minimization of the Boltzmann-Shannon entropy functional $F_\lambda$ over the space of probability measures.
In this paper's context, we can view such optimization dynamics as the (HK-Gaussian) gradient flow of the functional $F_\lambda$,
i.e., considering the gradient system $(\calP(\R^d),\calH_\rmB(\cdot|\target) , \SHK_{\alpha,\beta})$.

\paragraph*{Gaussian target measures}
In the first experiment, we consider the case when the target measure $\target$ is a Gaussian measure.
Without loss of generality, we set the target mass $\kappa_\target=1$.
We apply Algorithm~\ref{alg:JKO}
to simulate the HK-Gaussian gradient flow of the functional $\calH_\rmB(\mu |\target)$ over the space of positive measures.
We plot the evolution of the contour of a two-dimensional Gaussian measure $\mu_k$ in
\autoref{fig:gaussian_target}.
\begin{figure}[htbp]
  \centering
  \includegraphics[width=0.485\textwidth]{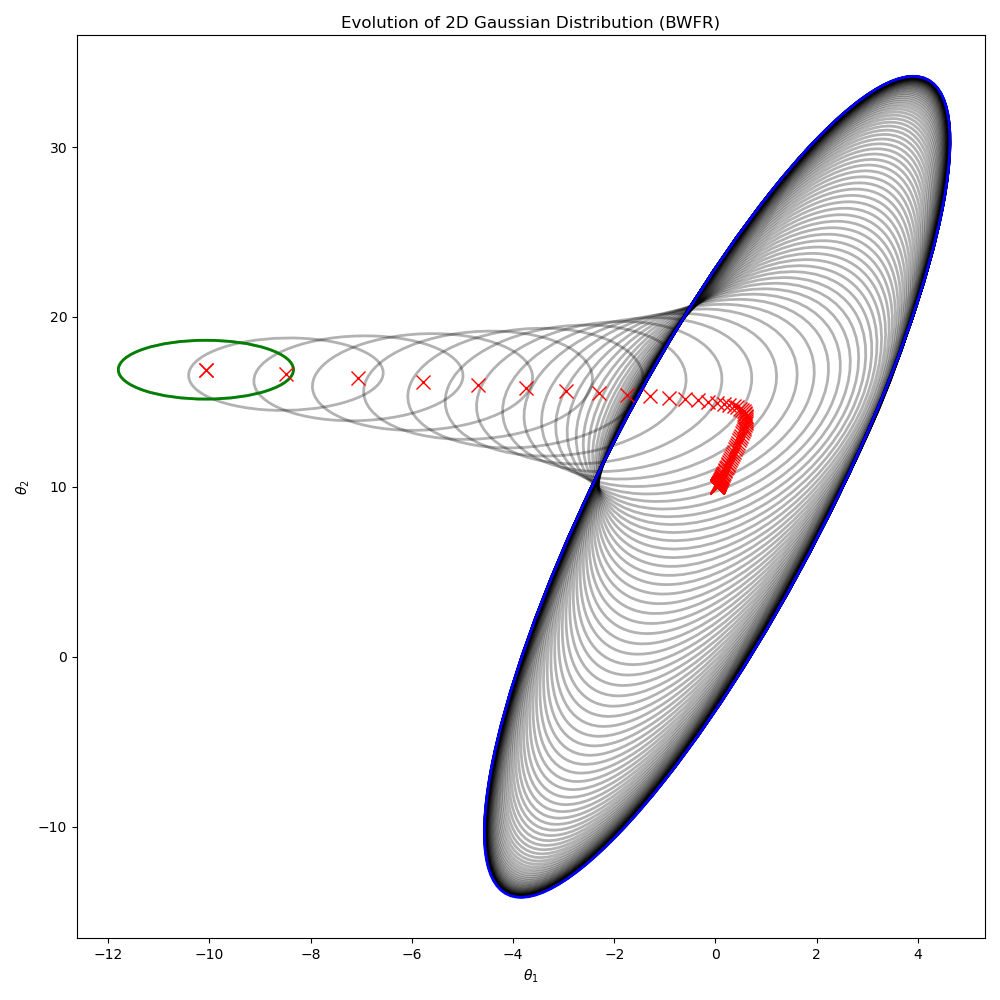}
  \includegraphics[width=0.97\textwidth]{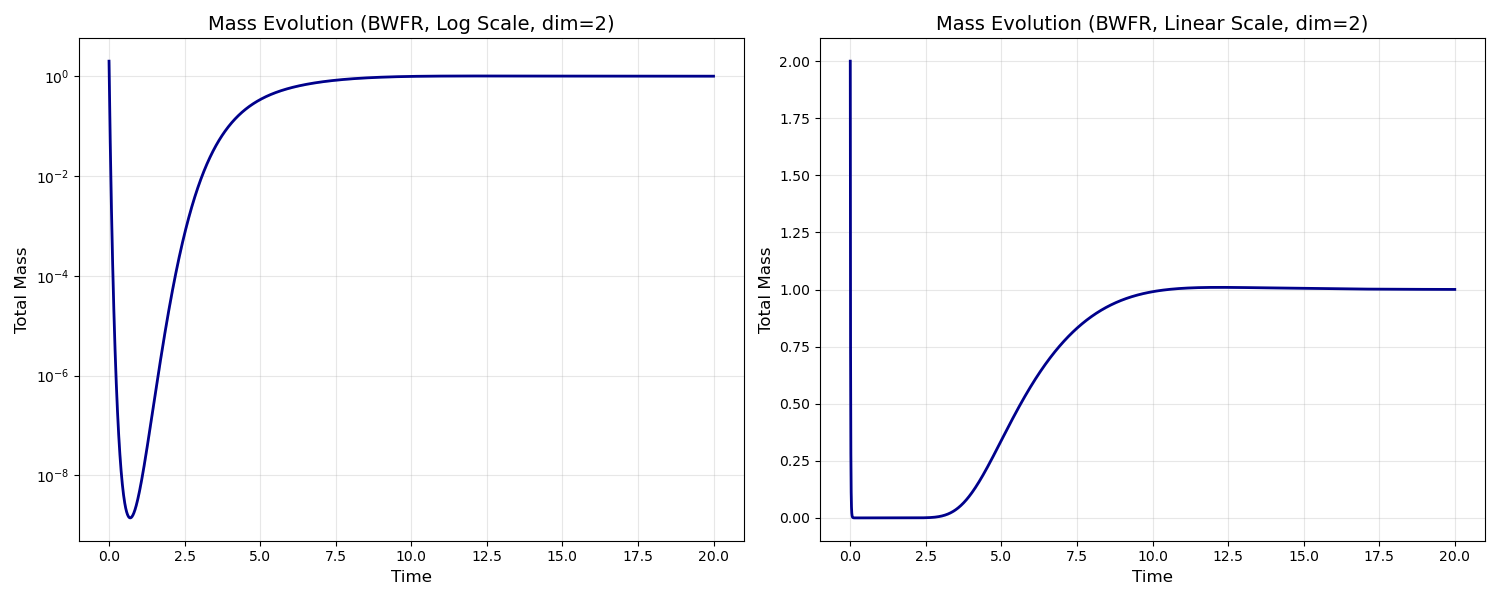}
  \caption{(Top) Evolution of the Gaussian probability measure $\mu_k$ (black contours) towards the target Gaussian measure $\target$ (blue contour).
  We simulated the evolution using Algorithm~\ref{alg:JKO} that implements the $\HK$-Gaussian gradient descent.
  (Bottom) Evolution of the mass variable $\kappa_k$ in both log and linear scale.}
  \label{fig:gaussian_target}
\end{figure}

We then consider the gradient flow of Gaussian measures defined on $\bbR^d$ with $d=100$.
In this case, the target and initial measures have less overlapping mass than the lower-dimensional setting.
\autoref{fig:kl_evo_dim100} plots the decay of the functional $\calH_\rmB(\mu_k|\target)$ and the mass variable $\kappa_k$ for the gradient flow of probability measures defined on $\bbR^{100}$.
\begin{figure}[htbp]
  \centering
  \includegraphics[width=0.97\textwidth]{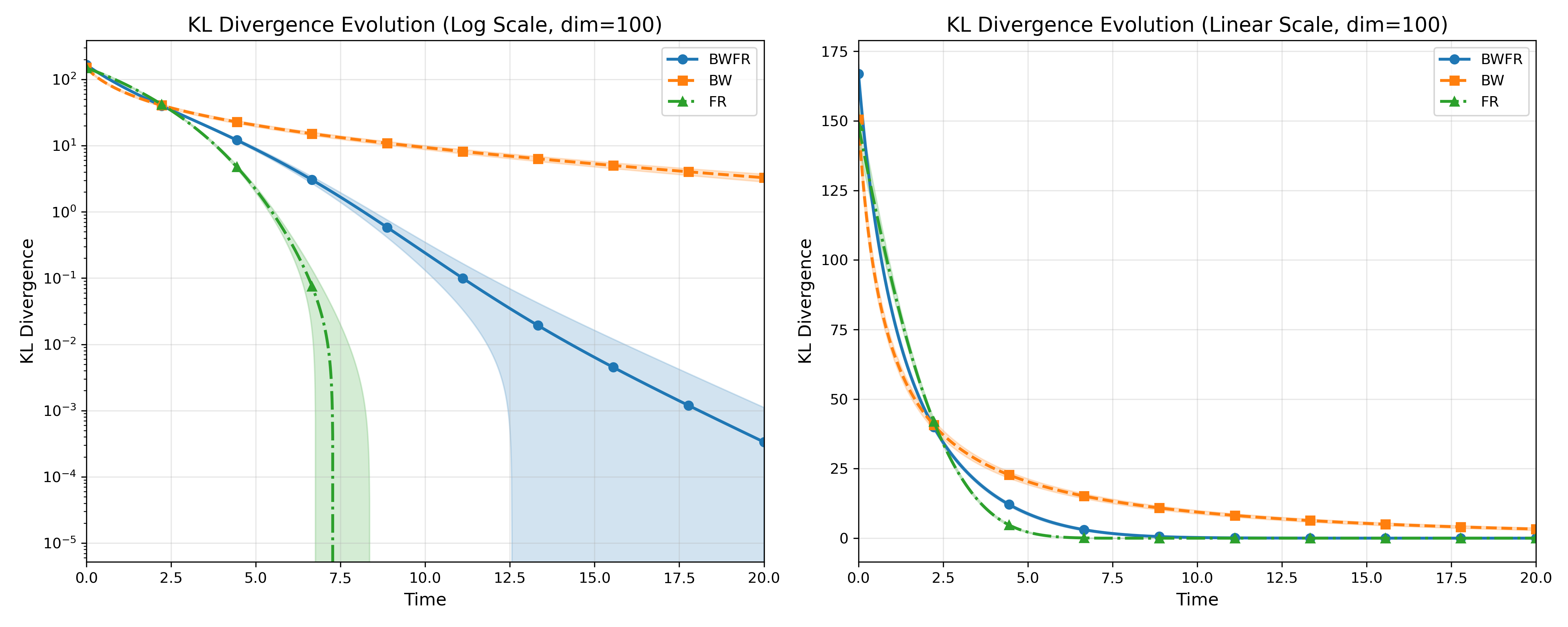}
  \includegraphics[width=0.97\textwidth]{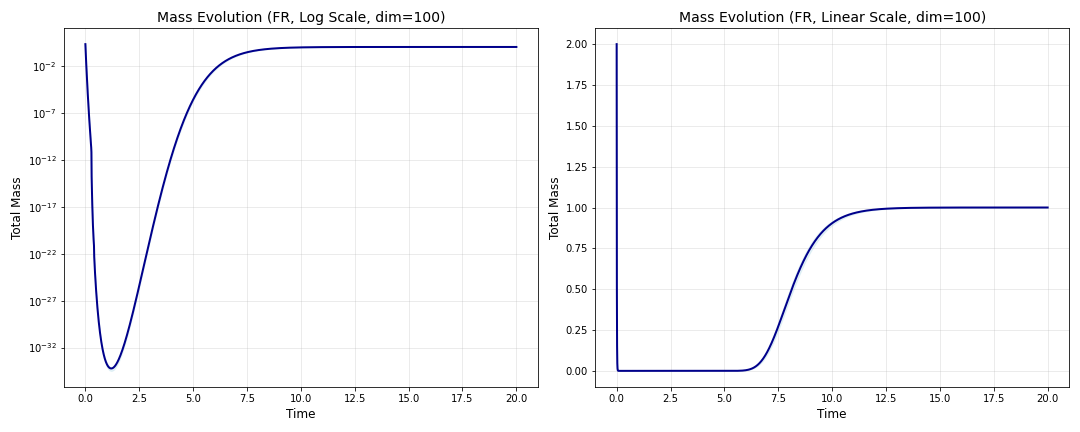}
  \caption{(Top) Evolution of the KL divergence $\calH_\rmB(\mu_k|\target)$ (blue line) for the three different gradient flows of probability measures defined on $\bbR^{100}$: the Hellinger-Kantorovich-Gaussian gradient flow (blue; abbreviated as BWFR (Bures-Wasserstein-Fisher-Rao)), the Fisher-Rao gradient flow (green), and the Bures-Wasserstein gradient flow (orange).
  We observe that the Fisher-Rao gradient flow has the most rapid convergence in KL divergence asymptotically, while the Bures-Wasserstein gradient flow is initially faster. The HK-Gaussian gradient flow is between those two.
  (Bottom) Evolution of the mass variable $\kappa_k$ for the gradient flow of probability measures defined on $\bbR^{100}$.
  }
  \label{fig:kl_evo_dim100}
\end{figure}

\paragraph*{Bayesian logistic regression}
As an example of a machine learning problem where the target (i.e. posterior) measure is not Gaussian,
we now visit a standard example from the machine learning literature, namely the Bayesian logistic regression task \cite{bishop2006pattern}.
In this case, we consider a two-class classification problem (as a risk minimization problem~\eqref{eq:risk_minimization}) given training data $\xi = (x, y)$ where $x\in \bbR^d$ is the feature vector and $y\in \{0,1\}$ is the binary response.
For each training data point $(x, y)$, we can define the loss function
using the negative log-likelihood of a Bernoulli random variable model,
\begin{align*}
  \ell_{\text{logistic}} (\theta, [x, y]) &= 
     -y \log \frac{\ee^{\theta^\top x+\theta_0}}{1+\ee^{\theta^\top x+\theta_0}}
    -(1-y) \log \frac{1}{1+\ee^{\theta^\top x+\theta_0}}\\
    &= - y \cdot \left(
      \theta^\top x + \theta_0\right)
      + \log \left(1 + \ee^{\theta^\top x + \theta_0}\right),
\end{align*}
where $\xi = (x, y)$ is a training data point.
Note that the logistic loss function $\ell_{\text{logistic}} (\theta, [x, y])$ is convex in $\theta$.
The task of Bayesian logistic regression is then to approximate a target posterior distribution $\target$, which does not have a closed-form expression and is non-Gaussian.
Furthermore, using our insight of the shape-mass relation of the $\HK$ and $\SHK$ flows, we omit the mass update step in Algorithm~\ref{alg:JKO} for the Bayesian logistic regression task.

In machine learning, the true data generating process $P_\text{data}$ of the risk minimization problem \eqref{eq:risk_minimization} is unknown in practice and replaced by the empirical distribution of the training data $\frac1n\sum_{i=1}^n \delta_{\xi_i}$.
Then, the Bayesian logistic regression problem amounts to the risk minimization problem
\begin{align*}
  \min_{\mu\in \mfG_1}
  \int 
  \frac1{n\lambda}\sum_{i=1}^n
        { \ell_{\text{logistic}} (\theta, [x_i, y_i])}
          \dd \mu(\theta) 
          + \int  \log \mu(\theta) \dd \mu (\dd \theta ) 
        ,
\end{align*}
where $x_i$ and $y_i$ are the training data pairs.
Using the previous notation,
we can write the functional as
the Boltzmann entropy functional
$$
\calH_\rmB\left(\mu(\theta)\Bigg| 
  \frac1{z}
          \exp{\left(-{ \frac1{n\lambda}\sum_{i=1}^n \ell_{\text{logistic}} (\theta, [x_i, y_i])}
          \right)}            
        \right)
        .
$$
We now apply Algorithm~\ref{alg:JKO} to simulate the HK-Gaussian gradient flow of the above entropy functional (without the mass update step). The evolution of the Gaussian probability measure $\rho_k$ is plotted in \autoref{fig:bayesian_logistic}~(Left).
In the right plot of \autoref{fig:bayesian_logistic}, we generate samples from the optimized Gaussian measure $\theta^*_i\sim \mu^*$ and plot the classification decision boundary associated with the parameter $\theta^*_i$. See the caption for more details.
\begin{figure}[htbp]
  \centering
  \includegraphics[width=0.49\textwidth]{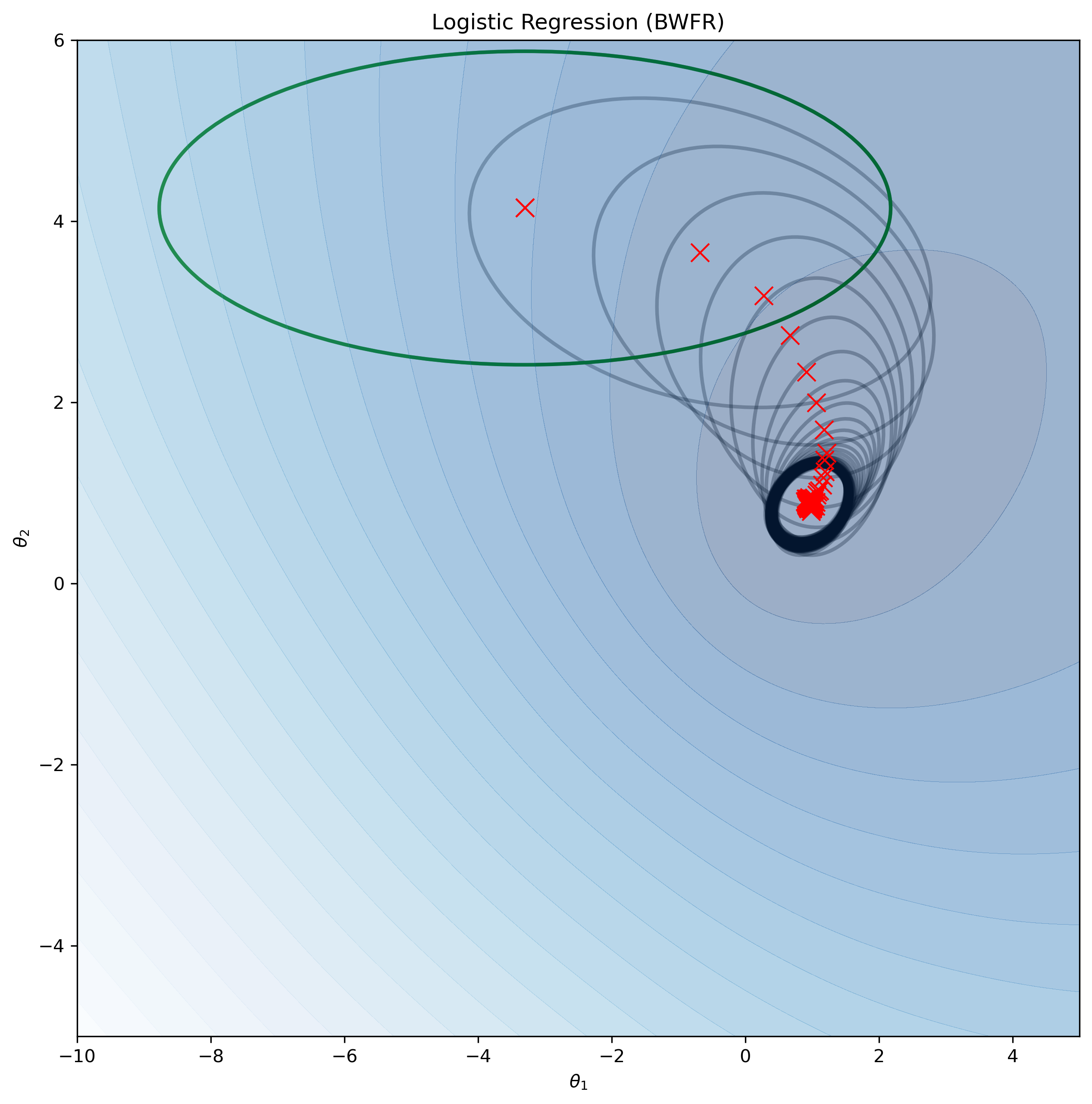}
  \includegraphics[width=0.49\textwidth]{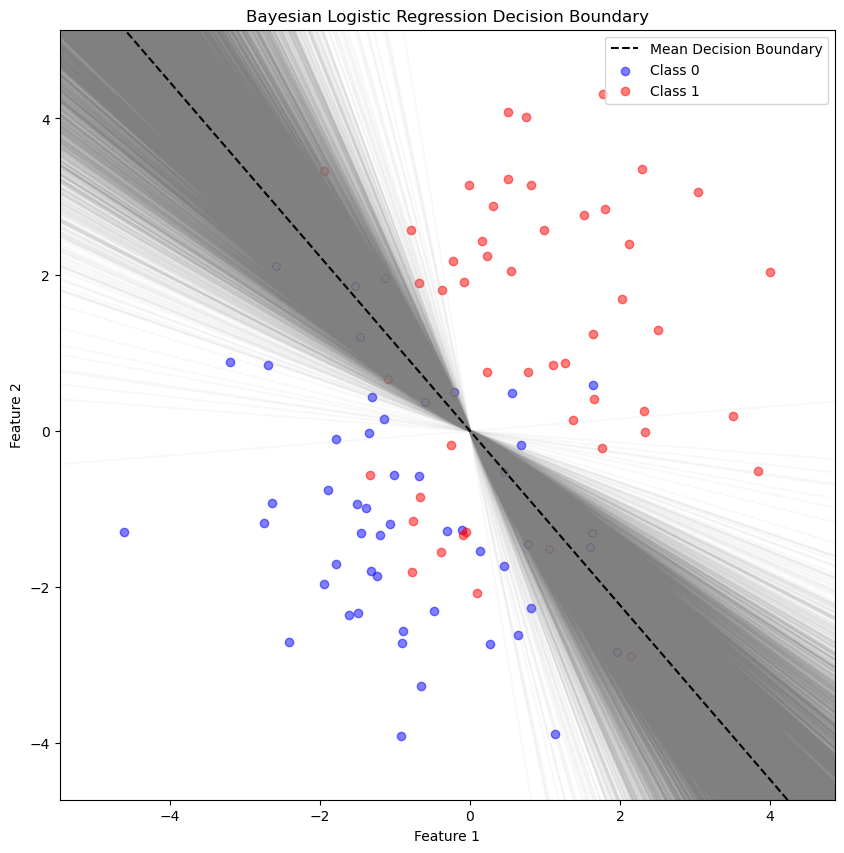}
  \caption{(Left) Evolution of the Gaussian probability measure $\rho_k$ (black ellipsoidal contours) for the Bayesian logistic regression task.
  The background color represents the log-density of the Gaussian measure. The darker shade corresponds to higher density.
  The contours of the true target measure $\target$ is not ellipsoidal in the Bayesian logistic regression task.
  (Right) Samples of the classification decision boundary from the optimized Gaussian measure.
  }
  \label{fig:bayesian_logistic}
\end{figure}

\appendix

\section{Useful formulas for Gaussians}
\label{se:AppGauss}

For constants $c_0,c_2,c_4 \in \R$,  vectors $a,b\in \R^d$ and symmetric
matrices $A,B\in \R^{d\ti 
  d}_\text{sym}$ we have the relations
\begin{subequations}
\label{eq:A.all}
\begin{gather}
\label{eq:A.1D}
\int_\R (c_0{+}c_2 x^2{+}c_4x^4)\, \frac{\ee^{-\frac12 x^2}}{\sqrt{2\target}} \dd
x=c_0{+}c_2{+}3c_4,   
\qquad \int_{\R^d} \frac{\ee^{-\frac12|x|^2}}{(2\target)^{d/2}}\dd x=1, 
\\[0.5em]
\label{eq:A.DD02}
\int_{\R^d} (a{\cdot}x)\,(b{\cdot}
  x)\,\frac{\ee^{-\frac12|x|^2}}{(2\target)^{d/2}}\dd x= a{\cdot} b,  
\qquad \int_{\R^d} x {\cdot} Ax \, 
   \frac{\ee^{-\frac12|x|^2}}{(2\target)^{d/2}} \dd x=\trace(A), 
 \\[0.5em]
\label{eq:A.DD4}
\int_{\R^d}(x{\cdot} Ax) \,(x{\cdot}Bx) \, \frac{\ee^{-\frac12|x|^2}} 
   {(2\target)^{d/2}}\dd x = 2\,A \mdot B +
   \trace(A)\trace(B). 
\end{gather}
\end{subequations}

\section{Well-posedness of \eqref{eq:ODEs.mfQ}}
\label{se:ODEs.mfQ}

Notice that the equation for the \emph{precision matrix} $A$ is independent of the other variables, allowing us to analyze \eqref{eq:ODEs.mfQ.A} independently. To this end, we consider \eqref{eq:ODEs.mfQ.A} as an evolution on the closed convex subset $\R^{d\ti d}_\spd$ of symmetric and positive semidefinite matrices, and we equip $\R^{d\ti d}$ with the scalar product
	\[
		\langle A,B\rangle := \trace[A^\top B],\qquad A,B\in \R^{d\ti d}.
	\]
	Defining the nonlinear operator 
	\[
		\R^{d\ti d}\ni A\mapsto f(A) := -\alpha \{A-\ol A,A\} - \beta (A-\ol A),
	\]
	where $\{A,B\}=AB + BA$ is the anti-commutator of $A$ and $B$, \eqref{eq:ODEs.mfQ.A} then reads $\dot A = f(A)$.

For any $A,B\in \R^{d\ti d}_\spd$, we see that
\[
	f(A)-f(B) = \{A+B-\ol A,A-B\},
\]
and from which we deduce
\[
	\langle f(A)-f(B),A-B\rangle = -2\alpha\, \trace[(A+B-\ol A)(A-B)^2] -\beta\,\trace[(A-B)^2].
\]
Since $A,B\in \R^{d\ti d}_\spd$, $\trace[(A+B)(A-B)^2] \ge 0$, and therefore,
\[
		\langle f(A)-f(B),A-B\rangle \le \bigl(2\alpha \|\ol A\|_\infty - \beta\bigr) \|A-B\|^2 =: \omega \|A-B\|^2,
\]
i.e., the nonlinear operator $f-\omega I$ is dissipative. An application of \cite[Theorem~6.1]{Martin1976} then provides a unique global solution to the initial value problem
	\[
		\dot A = f(A),\qquad A(0) = A_0\in \R^{d\ti d}_\spd.
	\]
	
	Notice that since $f(\ol A)=0$, we have that
	\[
		\langle f(A)-f(\ol A),A-\ol A\rangle = -2\alpha\, \trace[A(A-\ol A)^2] -\beta\,\trace[(A-\ol A)^2] \le -\beta  \|A-\ol A\|^2,
	\]
	from which we obtain the (coarse) exponential decay estimate
	\[
		\|A(t)-\ol A\|^2 \le \ee^{-\beta t}\|A_0-\ol A\|^2.
	\]
	Clearly, the decay rate can be improved if a uniform lower bound for the minimal eigenvalue of $A(t)$ along the flow is provided. This is performed in Section~\ref{se:LongTime}.
	
	\begin{remark}
	    There is also a formal gradient structure for \eqref{eq:ODEs.mfQ.A} with quadratic driving energy. Indeed, considering the driving energy
     \[
        \mathscr{E}(A) = \trace[(A-\ol A)^2],\qquad A\in\R^{d\ti d}_\spd,
     \]
     and dual dissipation potential
     \[
        \mathscr{R}^*(A,\zeta) = \frac{\alpha}{2} \trace[\zeta^2 A] + \frac{\beta}{4}\trace[\zeta^2],\qquad A\in \R^{d\ti d}_\spd,\;\zeta\in\R^{d\ti d},
     \]
     the right-hand side of \eqref{eq:ODEs.mfQ.A} reads
     \[
        f(A) = \partial_2 \mathscr{R}^*(A, - \rmD\mathscr{E}(A)).
     \]
	\end{remark}

\paragraph*{Acknowledgments.} A.M. was partially supported by
the Deutsche Forschungsgemeinschaft (DFG) through the Berlin Mathematics
Research Center MATH+ (EXC-2046/1, project ID: 390685689) subproject ``DistFell''.
J.Z. acknowledges the support from the Deutsche Forschungsgemeinschaft (DFG, German Research Foundation) as part of the priority programme "Theoretical Foundations of Deep Learning" (project number: 543963649). M.L. acknowledges support from DFG CRC/TRR 388 Rough
Analysis, Stochastic Dynamics and Related Topics, project A02. O.T. was supported by the Netherlands Organisation for Scientific Research (NWO) under Grant No.\ 680.92.18.05 and NGF.1582.22.009.

\footnotesize

\addcontentsline{toc}{section}{References}

\bibliographystyle{my_alpha}
\bibliography{alex_pub,bib_alex,bib_additional}

\end{document}